\documentclass[12pt,letterpaper]{amsart}
\usepackage{tikz}
\usepackage{array}
\usepackage{xcolor}
\usepackage{tikz-cd}
\usepackage{hyperref}
\usepackage{cleveref}
\usepackage{rotating}
\usepackage{rotating}
\usetikzlibrary{mindmap,trees,positioning}
\usepackage{eucal}

\usepackage[T1]{fontenc}
\usepackage[utf8]{inputenc}
\usepackage{lmodern}
\usepackage{amsmath, amssymb, amsfonts, amscd}

\usepackage{graphicx}
\usepackage{subcaption}
\usepackage{booktabs}
\usepackage[all]{xy}

\usepackage{geometry}
\newtheorem{theorem}{Theorem}[section]
\newtheorem{lemma}[theorem]{Lemma}
\theoremstyle{definition}
\newtheorem{definition}[theorem]{Definition}
\newtheorem{example}[theorem]{Example}
\newtheorem{proposition}[theorem]{Proposition}

\theoremstyle{remark}
\newtheorem{remark}[theorem]{Remark}

\DeclareMathOperator{\Per}{Per}

\newcommand{\Roll}{\text{\rm \bf Roll}}
\newcommand{\Adm}{\text{\rm \bf Adm}}

\newcommand{\R}{{\mathbb{R}}}

\newcommand{\rank}{\text{\rm rank}}
\newcommand{\conv}{\text{\rm conv}}
\newcommand{\graphscale}{0.45}

\newcounter{algorithm}

\definecolor{wiseblue}{rgb}{0.14,0.28,0.48}
\definecolor{deepteal}{rgb}{0.10,0.32,0.36}
\definecolor{aquamed}{rgb}{0.00,0.55,0.65}
\definecolor{softgrey}{rgb}{0.35,0.35,0.35}
\definecolor{col0}{RGB}{109,166,206}
\definecolor{col1}{RGB}{255,171,98}
\definecolor{col2}{RGB}{117,193,117}
\definecolor{col3}{RGB}{228,114,115}
\definecolor{col4}{RGB}{185,156,212}
\definecolor{col5}{RGB}{180,145,138}
\definecolor{col6}{RGB}{236,166,215}
\definecolor{col7}{RGB}{171,171,171}
\definecolor{col8}{RGB}{211,212,111}
\definecolor{col9}{RGB}{104,212,223}

\hypersetup{
  colorlinks=true,
  linkcolor=wiseblue,
  citecolor=deepteal,
  urlcolor=aquamed,
  pdfauthor={Jos\'e Ayala Hoffmann},
  pdftitle={Intrinsic Variational Geometry in Hard Disk Clusters}
}

\usepackage{etoolbox}
\makeatletter

\pretocmd{\@tocline}{\color{wiseblue}}{}{}
\apptocmd{\@tocline}{\color{black}}{}{}
\makeatother

\makeatletter
\def\@biblabel#1{[#1]}
\makeatother

\begin{document}
\title{Intrinsic Geometry of Hard Disk Clusters}

\author{Jos\'e Ayala Hoffmann}
\address{Universidad de Tarapac\'a, Iquique, Chile}
\email{jayalhoff@gmail.com}

\author{Fabi\'an Henry Vilaxa}
\address{Universidad de Tarapac\'a, Iquique, Chile}
\email{fabian.henry.vilaxa@alumnos.uta.cl}

\subjclass[2020]{Primary 52C15; Secondary 52A40, 49Q10, 51M25}

\keywords{disk packing, hard disks, perimeter minimisation, convex hull, contact graphs, rigidity}

\begin{abstract}
\baselineskip=20 true pt
\baselineskip=1.1\normalbaselineskip
Put \(n\) identical coins on a table with no two overlapping. Which arrangement makes the perimeter of the convex hull of the cluster as small as possible? Despite its elementary statement, the solution of this problem is known only up to four disks.

We produce a calculus for hard disk clusters of arbitrary finite size, providing class criticality conditions, first order descent tests, and second order spectral criteria for perimeter minimisation. A central difficulty is that the perimeter formula changes with the hull combinatorics, while the admissible first order geometry changes with the realised contacts. Our approach is guided by the principle that the realised geometry intrinsically determines both the local form of the functional and the admissible motions.

As an application, this article takes the first step beyond four disks by providing a solution for the five disk case. The minimum perimeter is \(10+2\pi\), attained in exactly three realised classes. Two admit perimeter preserving flexes, of dimensions one and two, while the third is rigid modulo rigid motions. The first order theory provides
pruning criteria, and the reduced admissible space together with its
intrinsic Hessian distinguish rigidity, second order instability, and
perimeter flat degeneracy. 
\end{abstract}
\maketitle

\baselineskip=20 true pt
\baselineskip=1.1\normalbaselineskip

\section{Prelude}

What is the smallest possible perimeter of the convex hull of a cluster of $n$ non overlapping unit disks in the plane?

Despite its elementary formulation, the problem quickly becomes nontrivial beyond four disks. By the Minkowski-Steiner formula, this problem is equivalent to minimising the perimeter of the convex hull of \(n\) points whose pairwise distances are at least two. The difficulty, however, is not the passage from disks to centres, but the geometry of the feasible space itself. Two independent combinatorial phenomena destroy global smoothness. First, the active non overlap constraints change as contacts are created or broken. Second, the local perimeter formula changes whenever the hull vertex set or its cyclic order changes. Consequently, neither the admissible first order geometry nor the perimeter functional can be described globally by a single smooth constrained problem.

In bar--joint rigidity theory, one commonly starts with an abstract graph,
prescribes its edges as bars, imposes fixed edge length equations, and then
differentiates those equations to obtain the infinitesimal rigidity equations
and the associated rigidity operator
\cite{asimow1978,connelly1982,whiteley1997}. A main conceptual point of this article is that the primary object is the realised non overlap geometry, rather than a prescribed abstract contact framework. The realised contact normals determine the local first order admissible structure, while the contact operator merely records this structure in matrix form.

At first order, the preservation of contacts depends only on the common tangent. If ${\bf c}_i$ and ${\bf c}_j$ are
the centres of two tangent disks and
$\delta{\bf c}_i,\delta{\bf c}_j$ are infinitesimal displacements, then the
realised contact imposes the first order non overlap condition
\[
\langle {\bf u}_{ij},
\delta{\bf c}_j-\delta{\bf c}_i\rangle\ge 0,
\qquad
{\bf u}_{ij}
=
\frac{{\bf c}_j-{\bf c}_i}
{\|{\bf c}_j-{\bf c}_i\|},
\]
while preservation of the realised tangency gives its equality case
\[
\langle {\bf u}_{ij},
\delta{\bf c}_j-\delta{\bf c}_i\rangle=0.
\]

The scalar
$
\langle {\bf u}_{ij},
\delta{\bf c}_j-\delta{\bf c}_i\rangle
$
has a direct geometric interpretation. It is the normal component of the relative velocity at the realised contact. Positive values open the contact to first order, zero preserves the
tangency to first order, and negative values are incompatible with first
order non overlap. The realised contact normals determine the local one sided admissible cone
\(\Adm({\bf c})\) and, within it, the rolling space
\(\Roll({\bf c})\), consisting of the infinitesimal motions that preserve
every contact to first order. The contact operator \(A({\bf c})\), equivalent to the rigidity operator of
the realised contact framework, represents the rolling conditions, with
\(\ker A({\bf c})=\Roll({\bf c})\). The novelty lies in the surrounding realised geometry and its variational interpretation.

The approach is intrinsic in both its coordinate free formulation and its
choice of functional. We impose no auxiliary energy. Instead, we study a
canonical geometric quantity of the cluster itself, the perimeter of its
convex hull. In our settings, the objective and the admissible first order directions are determined by the geometry of the cluster itself.

Smooth calculus becomes legitimate only after both the realised contacts
and the cyclic hull structure are fixed. On a fixed hull and contact class
the perimeter has a smooth local expression, and at a regular relative
interior point the tangent space of the class is precisely the rolling
space. The first variation then tests class criticality, where the perimeter
gradient must lie in the image of \(A({\bf c})^{\top}\), yielding the
multiplier equation. A first order descent
direction in the rolling space excludes a configuration immediately, while
the hull leaf obstruction provides an independent combinatorial exclusion
criterion.

The second variation at a class critical configuration, restricted to the
reduced rolling space obtained from \(\Roll({\bf c})\) by removing
infinitesimal rigid motions, defines the intrinsic Hessian. If this
reduced space is trivial, then every infinitesimal rolling motion is induced
by a rigid motion. Otherwise, a negative eigenvalue of the intrinsic Hessian
certifies second order instability, while a nontrivial kernel of the
intrinsic Hessian signals second order degeneracy.

Local minimality within a fixed hull and contact class does not immediately imply local minimality in the full configuration space, where contacts may open or form and the hull type may change. Ambient minimality therefore requires, in addition to the intrinsic class analysis, one sided no descent checks on the incident classes governing exits from the class.

Unlike asymptotic packing problems, finite disk problems depend stronlgy on boundary geometry and realised contact structure, as observed by Hopkins, Stillinger, and Torquato \cite{hopkinsstillingertorquato}.
Classical finite packing inequalities are due to Oler, Groemer, Wegner, and
Fejes T\'oth \cite{oler1961,groemer1963,wegner1986,fejestothbook}. Closest
to our setting, Sch\"urmann formulated the minimum perimeter problem for
finite disk clusters and proved structural restrictions on extremisers
\cite{schurmann2002}, while Kallrath, Ryu, Song, Lee, and Kim gave high
accuracy numerical solutions for minimum convex hulls of disks
\cite{kallrath2021}. Rigidity of sticky discs was studied by Connelly, Gortler, and Theran \cite{connellygortlertheran}, and the topology of confined disk configuration spaces by Baryshnikov, Bubenik, and Kahle and by Carlsson, Gorham, Kahle, and Mason \cite{baryshnikovbubenikkahle,carlssongorhamkahlemason}. The finite disk perimeter problem also arises in geometric knot theory through
thick tubes and their disk cross sections, as studied by Ayala and Hass
\cite{ayalahass}.

The paper is organised as follows. Section~\ref{sec:config-space} sets up
the hard disk configuration space, the Minkowski reduction, and the
existence of minimisers. Section~\ref{sec:fixed-class} develops the
realised contact geometry, the rolling space, fixed hull and contact
classes, and the multiplier criterion for class criticality.
Section~\ref{sec:obstructions} introduces the first variation and proves the hull leaf obstruction, while
Section~\ref{sec:second-variation} develops the intrinsic second variation
and its reduced diagnostics. In Section~\ref{sec:n5}, Theorem~\ref{thm:n5-final} we solve the minimum perimeter problem for five
disks. Appendix~\ref{sec:examples-four-disks} applies the second order theory to
classical four disk configurations through explicit spectral examples. Appendix~\ref{app:computational-verification}
records the computational checks used in the five disk analysis. The computational
implementations, input data, and recorded output are in~\cite{repo}.

\section{The hard disk configuration space and the perimeter functional}
\label{sec:config-space}

\subsection{The configuration space $\mathcal D_n$}

A finite collection of pairwise non overlapping disks of common radius
\(r>0\) is called a hard disk configuration. Uniform scaling of the entire
configuration by a positive factor preserves contact relations and multiplies
the perimeter of its convex hull by the same factor. It suffices to
consider the case \(r=1\).

\begin{definition}
The configuration space of $n$ unit hard disks is
\[
\mathcal D_n=
\left\{
{\bf c}=({\bf c}_1,\dots,{\bf c}_n)\in(\R^2)^n:
\|{\bf c}_i-{\bf c}_j\|\ge 2 \text{ for all } i\ne j
\right\}.
\]
Each ${\bf c}\in\mathcal D_n$ represents the cluster of unit disks centred at
${\bf c}_1,\dots,{\bf c}_n$.
\end{definition}

A pairwise non overlap inequality
$
\|{\bf c}_i-{\bf c}_j\|\ge2
$
is said to be active at a configuration \({\bf c}\) if equality holds.
Equivalently, the corresponding pair of disks is in contact.

From the intrinsic point of view, the primitive object is the realised hard disk configuration, not an abstract graph prescribed in advance. Its contacts are the active non overlap constraints, while its realised hull data determine the local expression for the perimeter. The nonlinear inequalities defining $\mathcal D_n$ therefore govern the ambient feasible geometry. The contact graph and rolling space are derived from this geometry. The former records the realised contacts, while the latter consists precisely of the infinitesimal motions for which every active first order non overlap inequality holds with equality.

\begin{figure}[ht]
\centering 
\setlength{\tabcolsep}{8pt}
\renewcommand{\arraystretch}{1.5}
\tikzset{diskhull/.style={black!100,line width=.5pt}}
\begin{tabular}{cc} 
\begin{tikzpicture}[scale=\graphscale,line cap=round,line join=round]
\draw[col0,line width=1.2pt] (-2.000,0.000) -- (0.000,0.000);
\draw[col1,line width=1.2pt] (0.000,0.000) -- (2.000,0.000);
\fill[black!17,opacity=0.45] (-2.000,0.000) circle (1.000);
\fill[black!17,opacity=0.45] (0.000,0.000) circle (1.000);
\fill[black!17,opacity=0.45] (2.000,0.000) circle (1.000);
\draw[gray,line width=0.8pt] (-2.000,0.000) circle (1.000);
\draw[gray,line width=0.8pt] (0.000,0.000) circle (1.000);
\draw[gray,line width=0.8pt] (2.000,0.000) circle (1.000);
\draw[diskhull]
  (-2.000,-1.000) -- (2.000,-1.000)
  arc[start angle=-90,end angle=90,radius=1.000]
  -- (-2.000,1.000)
  arc[start angle=90,end angle=270,radius=1.000];
\filldraw[black] (-2.000,0.000) circle (0.05);
\filldraw[black] (0.000,0.000) circle (0.05);
\filldraw[black] (2.000,0.000) circle (0.05);
\end{tikzpicture}
\hspace{0.1cm}
&
\begin{tikzpicture}[scale=\graphscale,line cap=round,line join=round]
\draw[col0,line width=1.2pt] (-1.000,-0.577) -- (1.000,-0.577);
\draw[col1,line width=1.2pt] (-1.000,-0.577) -- (0.000,1.155);
\draw[col2,line width=1.2pt] (1.000,-0.577) -- (0.000,1.155);
\fill[black!17,opacity=0.45] (-1.000,-0.577) circle (1.000);
\fill[black!17,opacity=0.45] (1.000,-0.577) circle (1.000);
\fill[black!17,opacity=0.45] (0.000,1.155) circle (1.000);
\draw[gray,line width=0.8pt] (-1.000,-0.577) circle (1.000);
\draw[gray,line width=0.8pt] (1.000,-0.577) circle (1.000);
\draw[gray,line width=0.8pt] (0.000,1.155) circle (1.000);
\draw[diskhull]
  (-1.000,-1.577)
  -- (1.000,-1.577)
  arc[start angle=-90,end angle=30,radius=1.000]
  -- (0.866,1.655)
  arc[start angle=30,end angle=150,radius=1.000]
  -- (-1.866,-0.077)
  arc[start angle=150,end angle=270,radius=1.000];
\filldraw[black] (-1.000,-0.577) circle (0.05);
\filldraw[black] (1.000,-0.577) circle (0.05);
\filldraw[black] (0.000,1.155) circle (0.05);
\end{tikzpicture}
\vspace{0.1cm}
\end{tabular}

\begin{tabular}{ccccc}
\begin{tikzpicture}[scale=\graphscale,line cap=round,line join=round]
\draw[col0,line width=1.2pt] (-3.000,0.000) -- (-1.000,0.000);
\draw[col1,line width=1.2pt] (-1.000,0.000) -- (1.000,0.000);
\draw[col2,line width=1.2pt] (1.000,0.000) -- (3.000,0.000);
\fill[black!17,opacity=0.45] (-3.000,0.000) circle (1.000);
\fill[black!17,opacity=0.45] (-1.000,0.000) circle (1.000);
\fill[black!17,opacity=0.45] (1.000,0.000) circle (1.000);
\fill[black!17,opacity=0.45] (3.000,0.000) circle (1.000);
\draw[gray,line width=0.8pt] (-3.000,0.000) circle (1.000);
\draw[gray,line width=0.8pt] (-1.000,0.000) circle (1.000);
\draw[gray,line width=0.8pt] (1.000,0.000) circle (1.000);
\draw[gray,line width=0.8pt] (3.000,0.000) circle (1.000);
\draw[diskhull]
  (-3.000,-1.000) -- (3.000,-1.000)
  arc[start angle=-90,end angle=90,radius=1.000]
  -- (-3.000,1.000)
  arc[start angle=90,end angle=270,radius=1.000];
\filldraw[black] (-3.000,0.000) circle (0.05);
\filldraw[black] (-1.000,0.000) circle (0.05);
\filldraw[black] (1.000,0.000) circle (0.05);
\filldraw[black] (3.000,0.000) circle (0.05);
\end{tikzpicture}
\vspace{0.1cm}
&
\begin{tikzpicture}[scale=\graphscale,line cap=round,line join=round]
\draw[col0,line width=1.2pt] (-1.000,-1.000) -- (1.000,-1.000);
\draw[col1,line width=1.2pt] (1.000,-1.000) -- (1.000,1.000);
\draw[col2,line width=1.2pt] (1.000,1.000) -- (-1.000,1.000);
\draw[col3,line width=1.2pt] (-1.000,1.000) -- (-1.000,-1.000);
\fill[black!17,opacity=0.45] (-1.000,-1.000) circle (1.000);
\fill[black!17,opacity=0.45] (1.000,-1.000) circle (1.000);
\fill[black!17,opacity=0.45] (1.000,1.000) circle (1.000);
\fill[black!17,opacity=0.45] (-1.000,1.000) circle (1.000);
\draw[gray,line width=0.8pt] (-1.000,-1.000) circle (1.000);
\draw[gray,line width=0.8pt] (1.000,-1.000) circle (1.000);
\draw[gray,line width=0.8pt] (1.000,1.000) circle (1.000);
\draw[gray,line width=0.8pt] (-1.000,1.000) circle (1.000);
\draw[diskhull]
  (-1.000,-2.000)
  -- (1.000,-2.000)
  arc[start angle=-90,end angle=0,radius=1.000]
  -- (2.000,1.000)
  arc[start angle=0,end angle=90,radius=1.000]
  -- (-1.000,2.000)
  arc[start angle=90,end angle=180,radius=1.000]
  -- (-2.000,-1.000)
  arc[start angle=180,end angle=270,radius=1.000];
\filldraw[black] (-1.000,-1.000) circle (0.05);
\filldraw[black] (1.000,-1.000) circle (0.05);
\filldraw[black] (1.000,1.000) circle (0.05);
\filldraw[black] (-1.000,1.000) circle (0.05);
\end{tikzpicture}
&
\begin{tikzpicture}[scale=\graphscale,line cap=round,line join=round]
\draw[col0,line width=1.2pt] (-1.732,0.000) -- (0.000,1.000);
\draw[col1,line width=1.2pt] (0.000,1.000) -- (1.732,0.000);
\draw[col2,line width=1.2pt] (1.732,0.000) -- (0.000,-1.000);
\draw[col3,line width=1.2pt] (0.000,-1.000) -- (0.000,1.000);
\draw[col4,line width=1.2pt] (0.000,-1.000) -- (-1.732,0.000);
\fill[black!17,opacity=0.45] (-1.732,0.000) circle (1.000);
\fill[black!17,opacity=0.45] (0.000,1.000) circle (1.000);
\fill[black!17,opacity=0.45] (1.732,0.000) circle (1.000);
\fill[black!17,opacity=0.45] (0.000,-1.000) circle (1.000);
\draw[gray,line width=0.8pt] (-1.732,0.000) circle (1.000);
\draw[gray,line width=0.8pt] (0.000,1.000) circle (1.000);
\draw[gray,line width=0.8pt] (1.732,0.000) circle (1.000);
\draw[gray,line width=0.8pt] (0.000,-1.000) circle (1.000);
\draw[diskhull]
  (-2.232,-0.866)
  -- (-0.500,-1.866)
  arc[start angle=-120,end angle=-60,radius=1.000]
  -- (2.232,-0.866)
  arc[start angle=-60,end angle=60,radius=1.000]
  -- (0.500,1.866)
  arc[start angle=60,end angle=120,radius=1.000]
  -- (-2.232,0.866)
  arc[start angle=120,end angle=240,radius=1.000];
\filldraw[black] (-1.732,0.000) circle (0.05);
\filldraw[black] (0.000,1.000) circle (0.05);
\filldraw[black] (1.732,0.000) circle (0.05);
\filldraw[black] (0.000,-1.000) circle (0.05);
\end{tikzpicture}
&
\begin{tikzpicture}[scale=\graphscale,line cap=round,line join=round]
\draw[col0,line width=1.2pt] (-2.366,0.000) -- (-0.366,0.000);
\draw[col1,line width=1.2pt] (-0.366,0.000) -- (1.366,1.000);
\draw[col2,line width=1.2pt] (1.366,1.000) -- (1.366,-1.000);
\draw[col3,line width=1.2pt] (1.366,-1.000) -- (-0.366,0.000);
\fill[black!17,opacity=0.45] (-2.366,0.000) circle (1.000);
\fill[black!17,opacity=0.45] (-0.366,0.000) circle (1.000);
\fill[black!17,opacity=0.45] (1.366,1.000) circle (1.000);
\fill[black!17,opacity=0.45] (1.366,-1.000) circle (1.000);
\draw[gray,line width=0.8pt] (-2.366,0.000) circle (1.000);
\draw[gray,line width=0.8pt] (-0.366,0.000) circle (1.000);
\draw[gray,line width=0.8pt] (1.366,1.000) circle (1.000);
\draw[gray,line width=0.8pt] (1.366,-1.000) circle (1.000);
\draw[diskhull]
  (-2.625,-0.966)
  -- (1.107,-1.966)
  arc[start angle=-105,end angle=0,radius=1.000]
  -- (2.366,1.000)
  arc[start angle=0,end angle=105,radius=1.000]
  -- (-2.625,0.966)
  arc[start angle=105,end angle=255,radius=1.000];
\filldraw[black] (-2.366,0.000) circle (0.05);
\filldraw[black] (-0.366,0.000) circle (0.05);
\filldraw[black] (1.366,1.000) circle (0.05);
\filldraw[black] (1.366,-1.000) circle (0.05);
\end{tikzpicture}
\hspace{0.1cm}
\begin{tikzpicture}[scale=\graphscale,line cap=round,line join=round]
\draw[col0,line width=1.2pt] (-2.000,0.000) -- (0.000,0.000);
\draw[col1,line width=1.2pt] (0.000,0.000) -- (1.000,1.732);
\draw[col2,line width=1.2pt] (0.000,0.000) -- (1.000,-1.732);
\fill[black!17,opacity=0.45] (-2.000,0.000) circle (1.000);
\fill[black!17,opacity=0.45] (0.000,0.000) circle (1.000);
\fill[black!17,opacity=0.45] (1.000,1.732) circle (1.000);
\fill[black!17,opacity=0.45] (1.000,-1.732) circle (1.000);
\draw[gray,line width=0.8pt] (-2.000,0.000) circle (1.000);
\draw[gray,line width=0.8pt] (0.000,0.000) circle (1.000);
\draw[gray,line width=0.8pt] (1.000,1.732) circle (1.000);
\draw[gray,line width=0.8pt] (1.000,-1.732) circle (1.000);
\draw[diskhull]
  (-2.500,-0.866)
  -- (0.500,-2.598)
  arc[start angle=-120,end angle=0,radius=1.000]
  -- (2.000,1.732)
  arc[start angle=0,end angle=120,radius=1.000]
  -- (-2.500,0.866)
  arc[start angle=120,end angle=240,radius=1.000];
\filldraw[black] (-2.000,0.000) circle (0.05);
\filldraw[black] (0.000,0.000) circle (0.05);
\filldraw[black] (1.000,1.732) circle (0.05);
\filldraw[black] (1.000,-1.732) circle (0.05);
\end{tikzpicture}
\end{tabular}
\caption{Complete list of connected contact graph types for \(n=3\) and \(n=4\). For these cases, the minimum perimeter problem can be checked by hand using elementary planar geometry. The list sizes are \(13\) for \(n=5\), \(46\) for \(n=6\), \(162\) for \(n=7\), and \(715\) for \(n=8\). This growth suggests the computational character of the problem, once a finite reduction has been established. These values are entries of the connected penny graph sequence OEIS A085632 \cite{oeisA085632}.}
\label{fig:34}
\end{figure}
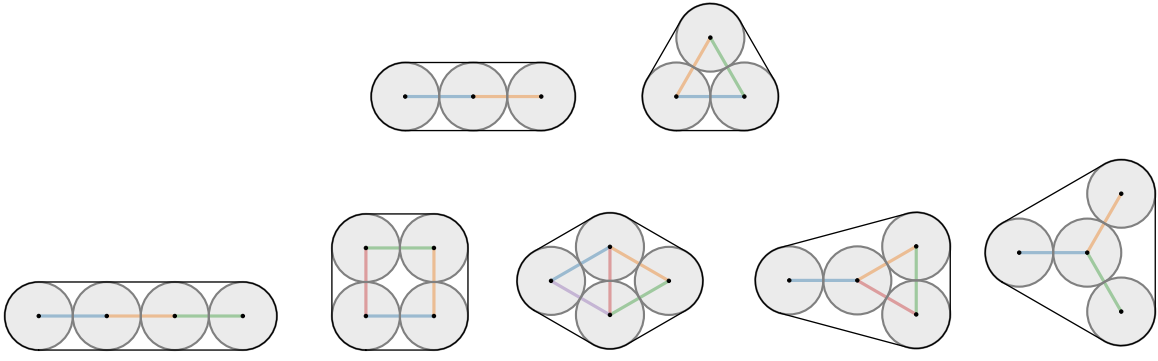

Note that $\mathcal D_n$ is closed but not compact since configurations may escape to infinity. After quotienting by rigid motions and imposing a uniform diameter bound, the resulting configuration space is compact.

\subsection{Cluster hull, centre hull, and the Minkowski shift}
Given ${\bf c}=({\bf c}_1,\dots,{\bf c}_n)$, let
\[
D_i=\{{\bf x}\in\R^2:\|{\bf x}-{\bf c}_i\|\le 1\}.
\]
Define the cluster hull
\[
h({\bf c})=\conv\left(\bigcup_{i=1}^n D_i\right)
\]
and the centre hull
\[
P({\bf c})=\conv\{{\bf c}_1,\dots,{\bf c}_n\}.
\]
Since $h({\bf c})$ is a compact convex subset of $\R^2$, its boundary
$\partial h({\bf c})$ is a rectifiable closed curve. We define the perimeter
of the hard disk cluster by
\[
\Per({\bf c})
=
\operatorname{Length}(\partial h({\bf c})).
\]
We emphasise that $\Per({\bf c})$ is the perimeter of the cluster hull, not the length of the boundary of the union of disks.

\begin{proposition}
\label{prop:minkowski-shift}
For every ${\bf c}\in\mathcal D_n$ one has
\[
h({\bf c})=P({\bf c})\oplus D,
\]
where $D$ is the unit disk centred at the origin. Consequently,
\[
\Per({\bf c})=\Per(P({\bf c}))+2\pi.
\]
\end{proposition}

\begin{proof}
Since
\[
\{{\bf c}_1,\dots,{\bf c}_n\}\oplus D=\bigcup_{i=1}^n D_i,
\]
taking convex hulls gives
\[
h({\bf c})
=
\conv\left(\bigcup_{i=1}^n D_i\right)
=
\conv\left(\{{\bf c}_1,\dots,{\bf c}_n\}\oplus D\right)
=
P({\bf c})\oplus D.
\]
The perimeter identity is the planar Steiner formula for outer parallel bodies.
\end{proof}

\begin{remark}[Continuity and local smoothness of the perimeter]
\label{rem:smooth}
The map ${\bf c}\mapsto P({\bf c})$ is continuous with respect to the Hausdorff topology on compact convex subsets of $\R^2$, and perimeter is continuous on planar convex bodies. Hence $\Per$ is continuous on $\mathcal D_n$.

Moreover, on any neighbourhood in $\mathcal D_n$ where the hull vertex set and its cyclic order are fixed, one has
\[
\Per({\bf d})=\sum_{(i,j)\in\mathcal B}\|{\bf d}_j-{\bf d}_i\|+2\pi
\]
for a fixed cyclic hull edge list $\mathcal B$. In particular, $\Per$ is a smooth function of the centres on such a neighbourhood.
\end{remark}

\begin{figure}[htbp]
\centering
\begin{tikzpicture}[scale=0.9, line cap=round, line join=round]
\tikzset{diskhull/.style={black!100,line width=.5pt}}
\fill[black!17, opacity=0.45] (-1.6,1.2) circle (1.000);
\fill[black!17, opacity=0.45] (-1.414,-1.414) circle (1.000);
\fill[black!17, opacity=0.45] (0,0) circle (1.000);
\fill[black!17, opacity=0.45] (1.732,1.0) circle (1.000);
\fill[black!17, opacity=0.45] (2.932,-0.6) circle (1.000);
\draw[gray, line width=0.8pt] (-1.6,1.2) circle (1.000);
\draw[gray, line width=0.8pt] (-1.414,-1.414) circle (1.000);
\draw[gray, line width=0.8pt] (0,0) circle (1.000);
\draw[gray, line width=0.8pt] (1.732,1.0) circle (1.000);
\draw[gray, line width=0.8pt] (2.932,-0.6) circle (1.000);
\draw[black!80, line width=0.8pt]
  (-1.6,1.2) -- (1.732,1.0) -- (2.932,-0.6) -- (-1.414,-1.414) -- cycle;
\draw[gray, line width=0.8pt] (-1.6,1.2) -- (-2.597,1.129);
\draw[gray, line width=0.8pt] (-1.6,1.2) -- (-1.540,2.198);
\draw[gray, line width=0.8pt] (1.732,1.0) -- (1.792,1.998);
\draw[gray, line width=0.8pt] (1.732,1.0) -- (2.532,1.6);
\draw[gray, line width=0.8pt] (2.932,-0.6) -- (3.732,0.0);
\draw[gray, line width=0.8pt] (2.932,-0.6) -- (3.116,-1.583);
\draw[gray, line width=0.8pt] (-1.414,-1.414) -- (-1.230,-2.397);
\draw[gray, line width=0.8pt] (-1.414,-1.414) -- (-2.411,-1.485);
\draw[diskhull]
  (-2.597,1.129) -- (-2.411,-1.485)
  arc[start angle=184.1,end angle=280.6,radius=1.000]
  -- (3.116,-1.583)
  arc[start angle=-79.4,end angle=36.9,radius=1.000]
  -- (2.532,1.6)
  arc[start angle=36.9,end angle=86.6,radius=1.000]
  -- (-1.540,2.198)
  arc[start angle=86.6,end angle=184.1,radius=1.000];
\draw[col0, line width=1.2pt] (-1.6,1.2) -- (0,0);
\draw[col1, line width=1.2pt] (-1.414,-1.414) -- (0,0);
\draw[col2, line width=1.2pt] (0,0) -- (1.732,1.0);
\draw[col3, line width=1.2pt] (1.732,1.0) -- (2.932,-0.6);
\filldraw[black] (-1.6,1.2) circle (0.05);
\filldraw[black] (-1.414,-1.414) circle (0.05);
\filldraw[black] (0,0) circle (0.05);
\filldraw[black] (1.732,1.0) circle (0.05);
\filldraw[black] (2.932,-0.6) circle (0.05);
\filldraw[fill=red, draw=black, line width=0.4pt] (-0.8,0.6) circle (0.05);
\filldraw[fill=red, draw=black, line width=0.4pt] (-0.707,-0.707) circle (0.05);
\filldraw[fill=red, draw=black, line width=0.4pt] (0.866,0.5) circle (0.05);
\filldraw[fill=red, draw=black, line width=0.4pt] (2.332,0.2) circle (0.05);
\node[anchor=north east, inner sep=4pt, font=\scriptsize] at (-1.6,1.2) {${\bf c}_1$};
\node[anchor=south east, inner sep=4pt, font=\scriptsize] at (-1.414,-1.414) {${\bf c}_2$};
\node[anchor=north west, inner sep=4pt, font=\scriptsize] at (0,0) {${\bf c}_3$};
\node[anchor=south east, inner sep=4pt, font=\scriptsize] at (1.732,1.0) {${\bf c}_4$};
\node[anchor=north east, inner sep=4pt, font=\scriptsize] at (2.932,-0.6) {${\bf c}_5$};
\end{tikzpicture}
\caption{The cluster hull is the Minkowski sum
\(h({\bf c})=P({\bf c})\oplus D\), so
\(\Per({\bf c})=\Per(P({\bf c}))+2\pi\). The centre hull, realised contacts,
and contact graph are shown. Note the exterior angles of the centre hull sum to
\(2\pi\).}
\label{fig:minkowski}
\end{figure}
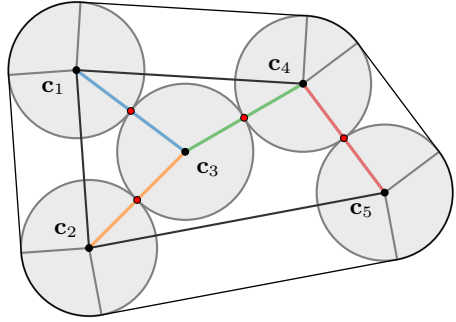

\begin{theorem}[Existence of minimisers]
\label{thm:existence-global-min}
For each fixed $n\ge 1$, the perimeter functional
\[
\Per:\mathcal D_n\to\R
\]
attains its minimum on $\mathcal D_n$.
\end{theorem}

\begin{proof}
Let $\{{\bf c}^m\}\subset\mathcal D_n$ be a minimizing sequence for $\Per$.
Using translation invariance, we may assume
\[
{\bf c}_1^m={\bf 0}
\qquad
\text{for all }m.
\]
After discarding finitely many terms, choose $L<\infty$ such that
\[
\Per({\bf c}^m)\le L
\qquad
\text{for all }m.
\]

The functional $\Per$ is continuous on $\mathcal D_n$. Indeed, the map
\[
{\bf c}\longmapsto P({\bf c})
=
\operatorname{conv}\{{\bf c}_1,\dots,{\bf c}_n\}
\]
is continuous with respect to the Hausdorff topology on compact convex sets,
and perimeter is continuous on planar convex bodies.

By Proposition~\ref{prop:minkowski-shift},
\[
\Per(P({\bf c}^m))
=
\Per({\bf c}^m)-2\pi
\le
L-2\pi.
\]

Since ${\bf c}_1^m={\bf 0}\in P({\bf c}^m)$, every centre ${\bf c}_i^m$ lies
in $P({\bf c}^m)$ and therefore satisfies
\[
\|{\bf c}_i^m\|
\le
\operatorname{diam}(P({\bf c}^m)).
\]

For any compact planar convex set $K$ one has
\[
\operatorname{diam}(K)\le \Per(K).
\] Hence
\[
\|{\bf c}_i^m\|
\le
\operatorname{diam}(P({\bf c}^m))
\le
\Per(P({\bf c}^m))
\le
L-2\pi
\qquad
\text{for all }i,m.
\]

Thus $\{{\bf c}^m\}$ is bounded in $(\R^2)^n$. Passing to a subsequence, we may
assume
\[
{\bf c}^m\to{\bf c}^*
\qquad
\text{in }(\R^2)^n.
\]
Since $\mathcal D_n$ is closed, ${\bf c}^*\in\mathcal D_n$. By continuity of
$\Per$,
\[
\Per({\bf c}^*)
=
\lim_{m\to\infty}\Per({\bf c}^m)
=
\inf_{{\bf c}\in\mathcal D_n}\Per({\bf c}).
\]
So ${\bf c}^*$ is a global minimiser.
\end{proof}

\section{Fixed hull contact classes and the rolling space}
\label{sec:fixed-class}

Given ${\bf c}\in\mathcal D_n$, its realised contact set is
\[
E({\bf c})=\{\{i,j\}: \|{\bf c}_i-{\bf c}_j\|=2\}.
\]

 The contact graph $G({\bf c})$ is the graph on $\{1,\dots,n\}$ whose edges
record these contacts. The contacts are determined by the active
inequalities of $\mathcal D_n$, while the graph is introduced afterwards as
derived data.

In the variational setting, the contact data are used together with the hull
data. On a fixed contact class, the active contacts become equality
constraints, while non contacts remain strict separation inequalities. The
hull data determine the local smooth expression for $\Per$, and the rolling
space records the contact preserving first order motions.

\begin{definition}
A configuration ${\bf c}\in\mathcal D_n$ is {\bf hull stable} if there exists an open
neighbourhood $U\subset(\R^2)^n$ of ${\bf c}$ such that for every
${\bf d}\in U\cap\mathcal D_n$ the centre hull $P({\bf d})$ has the same set of
hull vertex indices and the same cyclic order as $P({\bf c})$.
\end{definition}

\subsection{Contact classes, equality manifold and regularity}
For the purpose of describing a local contact class, fix a graph $G=(V,E)$ on
$\{1,\dots,n\}$. We write $E=E(G)$ for its edge set.
The contact class of $G$ is
\[
\mathcal C(G)=
\left\{
{\bf c}\in\mathcal D_n:
\|{\bf c}_i-{\bf c}_j\|=2 \ \forall \{i,j\}\in E(G),\quad
\|{\bf c}_i-{\bf c}_j\|>2 \ \forall \{i,j\}\notin E(G)
\right\}.
\]

A pairwise non overlap inequality is said to be active if it is satisfied with equality, equivalently if the corresponding pair of disks is in contact. Note the following distinction. $E(G)$ is the prescribed edge set of the abstract graph
$G$ used to describe the local class, while $E({\bf c})$ is the realised
contact set of a configuration ${\bf c}$. One has ${\bf c}\in\mathcal C(G)$ if
and only if $E({\bf c})=E(G)$.
Associated to $G$ is the equality manifold
\[
M(G)=
\left\{
{\bf d}\in(\R^2)^n:
\|{\bf d}_i-{\bf d}_j\|=2 \ \forall \{i,j\}\in E(G)
\right\}.
\]

For each edge \(\{i,j\}\in E(G)\), define
\[
g_{ij}({\bf d})=\|{\bf d}_j-{\bf d}_i\|^2-4.
\]
Define the constraint map
\[
F:(\R^2)^n\longrightarrow \R^{E(G)}
\]
by
\[
F({\bf d})
=
\bigl(g_{ij}({\bf d})\bigr)_{\{i,j\}\in E(G)}.
\]
Then
\[
M(G)=F^{-1}(0).
\]

We say that \({\bf d}\in M(G)\) is regular for \(G\) if the gradients
\[
\{\nabla g_{ij}({\bf d})\}_{\{i,j\}\in E(G)}
\]
are linearly independent in \(\R^{2n}\), equivalently if
\[
DF({\bf d}):(\R^2)^n\longrightarrow \R^{E(G)}
\]
has full row rank.

\begin{lemma}
\label{lem:manifold-rolling}
Let ${\bf c}\in M(G)$ be regular for $G$. Then $M(G)$ is a smooth embedded
submanifold of $(\R^2)^n$ near ${\bf c}$, of codimension $|E(G)|$.
\end{lemma}

\begin{proof}
Let $F=(g_{ij})_{\{i,j\}\in E(G)}$. Then $M(G)=F^{-1}(0)$. Since ${\bf c}$ is
regular for $G$, the gradients $\{\nabla g_{ij}({\bf c})\}_{\{i,j\}\in E(G)}$
are linearly independent, equivalently $DF({\bf c})$ has full row rank. By the
regular value theorem, $M(G)$ is a smooth embedded submanifold near ${\bf c}$,
of codimension $|E(G)|$.
\end{proof}

\begin{definition}
\label{def:fixed-hull-contact-class}
Let ${\bf c}\in\mathcal C(G)$ be hull stable and regular for $G$. Let
$\mathcal F=\mathcal F({\bf c})$ denote the local hull class determined by
${\bf c}$, namely the set of configurations whose centre hull has the same hull
vertex set and the same cyclic order as $P({\bf c})$. Let
$
\mathcal B=\mathcal B({\bf c})
$
denote the corresponding cyclic hull edge list.
A fixed hull and contact class near ${\bf c}$ is a set of the form
\[
X=\mathcal F\cap\mathcal C(G).
\]
\end{definition}

\begin{example}[Same contact graph, different hull cycles]
\label{ex:same-contact-different-hull}
A fixed hull and contact class is not determined by the contact graph alone. In Figure \ref{fig:same-contact-different-hull} we show two distinct realisations of the same five disk contact graph whose edges are
\[
E(G)=\{\{1,2\},\{2,3\},\{3,4\},\{3,5\},\{4,5\}\}
\]

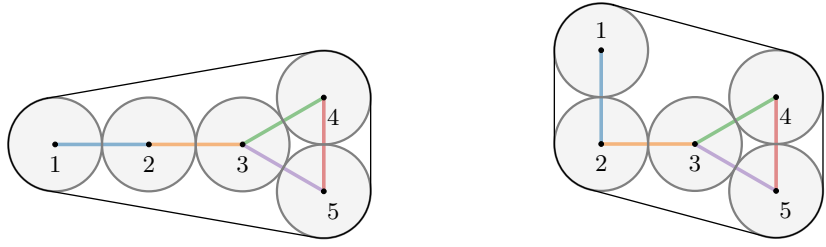
\begin{figure}[ht]
\centering
\tikzset{diskhull/.style={black!100,line width=.5pt}}
\begin{tabular}{c@{\hspace{2.4cm}}c}
\begin{tikzpicture}[scale=0.62, line cap=round, line join=round]
\coordinate (C1) at (-4.000,0.000);
\coordinate (C2) at (-2.000,0.000);
\coordinate (C3) at (0.000,0.000);
\coordinate (C4) at (1.732,1.000);
\coordinate (C5) at (1.732,-1.000);
\draw[col0,line width=1.35pt] (C1) -- (C2);
\draw[col1,line width=1.35pt] (C2) -- (C3);
\draw[col2,line width=1.35pt] (C3) -- (C4);
\draw[col3,line width=1.35pt] (C4) -- (C5);
\draw[col4,line width=1.35pt] (C3) -- (C5);
\fill[black!17,opacity=0.25] (C1) circle (1.000);
\fill[black!17,opacity=0.25] (C2) circle (1.000);
\fill[black!17,opacity=0.25] (C3) circle (1.000);
\fill[black!17,opacity=0.25] (C4) circle (1.000);
\fill[black!17,opacity=0.25] (C5) circle (1.000);
\draw[gray,line width=0.85pt] (C1) circle (1.000);
\draw[gray,line width=0.85pt] (C2) circle (1.000);
\draw[gray,line width=0.85pt] (C3) circle (1.000);
\draw[gray,line width=0.85pt] (C4) circle (1.000);
\draw[gray,line width=0.85pt] (C5) circle (1.000);
\draw[diskhull]
  (-4.1719,0.9851)
  -- (1.5601,1.9851)
  arc[start angle=99.896,end angle=0,radius=1.000]
  -- (2.7320,-1.0000)
  arc[start angle=0,end angle=-99.896,radius=1.000]
  -- (-4.1719,-0.9851)
  arc[start angle=-99.896,end angle=-260.104,radius=1.000];
\filldraw[black] (C1) circle (0.05);
\filldraw[black] (C2) circle (0.05);
\filldraw[black] (C3) circle (0.05);
\filldraw[black] (C4) circle (0.05);
\filldraw[black] (C5) circle (0.05);
\node[font=\scriptsize] at (-4.000,-0.42) {$1$};
\node[font=\scriptsize] at (-2.000,-0.42) {$2$};
\node[font=\scriptsize] at (0.000,-0.42) {$3$};
\node[font=\scriptsize] at (1.932,0.58) {$4$};
\node[font=\scriptsize] at (1.932,-1.42) {$5$};
\end{tikzpicture}
&
\begin{tikzpicture}[scale=0.62, line cap=round, line join=round]
\coordinate (C1) at (-2.000,2.000);
\coordinate (C2) at (-2.000,0.000);
\coordinate (C3) at (0.000,0.000);
\coordinate (C4) at (1.732,1.000);
\coordinate (C5) at (1.732,-1.000);
\draw[col0,line width=1.35pt] (C1) -- (C2);
\draw[col1,line width=1.35pt] (C2) -- (C3);
\draw[col2,line width=1.35pt] (C3) -- (C4);
\draw[col3,line width=1.35pt] (C4) -- (C5);
\draw[col4,line width=1.35pt] (C3) -- (C5);
\fill[black!17,opacity=0.25] (C1) circle (1.000);
\fill[black!17,opacity=0.25] (C2) circle (1.000);
\fill[black!17,opacity=0.25] (C3) circle (1.000);
\fill[black!17,opacity=0.25] (C4) circle (1.000);
\fill[black!17,opacity=0.25] (C5) circle (1.000);
\draw[gray,line width=0.85pt] (C1) circle (1.000);
\draw[gray,line width=0.85pt] (C2) circle (1.000);
\draw[gray,line width=0.85pt] (C3) circle (1.000);
\draw[gray,line width=0.85pt] (C4) circle (1.000);
\draw[gray,line width=0.85pt] (C5) circle (1.000);
\draw[diskhull]
  (-1.7412,2.9659)
  -- (1.9908,1.9659)
  arc[start angle=75.000,end angle=0,radius=1.000]
  -- (2.7320,-1.0000)
  arc[start angle=0,end angle=-105.000,radius=1.000]
  -- (-2.2588,-0.9659)
  arc[start angle=-105.000,end angle=-180.000,radius=1.000]
  -- (-3.0000,2.0000)
  arc[start angle=180.000,end angle=75.000,radius=1.000];
\filldraw[black] (C1) circle (0.05);
\filldraw[black] (C2) circle (0.05);
\filldraw[black] (C3) circle (0.05);
\filldraw[black] (C4) circle (0.05);
\filldraw[black] (C5) circle (0.05);
\node[font=\scriptsize] at (-2.000,2.42) {$1$};
\node[font=\scriptsize] at (-2.000,-0.42) {$2$};
\node[font=\scriptsize] at (0.000,-0.42) {$3$};
\node[font=\scriptsize] at (1.932,0.58) {$4$};
\node[font=\scriptsize] at (1.932,-1.42) {$5$};
\end{tikzpicture}
\end{tabular}
\caption{Realisations of the same contact graph with different cyclic hull
data: left, \((15)(54)(41)\); right, \((12)(25)(54)(41)\). Contact graph
edges are unordered, while cyclic hull data are recorded as ordered hull edges,
up to cyclic shift and reversal.}
\label{fig:same-contact-different-hull}
\end{figure}
The fixed hull condition is needed, since it determines the local perimeter formula.
\end{example}

\subsection{Prescribed lengths versus realised contacts}
\label{intrinsic}
Under bar--joint rigidity one starts with an abstract graph and prescribes the
edge length constraints. For an edge \(\{i,j\}\), the fixed length equation is
\[
\|{\bf p}_i-{\bf p}_j\|^2=\ell_{ij}^2.
\]
Differentiating gives
\[
\langle {\bf p}_i-{\bf p}_j,\delta{\bf p}_i-\delta{\bf p}_j\rangle=0,
\]
which describes infinitesimal motions preserving the prescribed edge length.

The hard disk setting has a different order of construction. The ambient
configuration space is defined first by the non overlap inequalities
\[
\|{\bf c}_i-{\bf c}_j\|\ge 2.
\]
A contact is a realised equality in this system,
\[
\|{\bf c}_i-{\bf c}_j\|=2,
\]
so an edge records an existing tangency rather than a primitive prescribed bar.

At a realised contact \(\{i,j\}\), the first order non overlap condition is
\begin{equation}
\label{eq:first-order-nonoverlap}
\langle {\bf u}_{ij},\delta{\bf c}_j-\delta{\bf c}_i\rangle\ge 0,
\qquad
{\bf u}_{ij}
=
\frac{{\bf c}_j-{\bf c}_i}{\|{\bf c}_j-{\bf c}_i\|}.
\end{equation}

Geometrically, the inner product measures the normal component of the relative
velocity of the two centres at the contact, see Figure~\ref{fig:contact}. A
positive value opens the contact, zero preserves it, and a negative value
violates the non overlap constraint. Thus \eqref{eq:first-order-nonoverlap} describes first order admissibility in
the ambient hard disk configuration space, while equality in
\eqref{eq:first-order-nonoverlap} describes contact preservation to first
order.

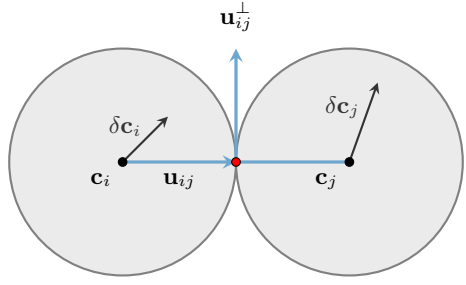
\begin{figure}[htbp]
\centering
\begin{tikzpicture}[scale=1.5, line cap=round, line join=round]
\fill[black!17, opacity=0.45] (-1,0) circle (1.000);
\draw[gray, line width=0.8pt] (-1,0) circle (1.000);
\fill[black!17, opacity=0.45] (1,0) circle (1.000);
\draw[gray, line width=0.8pt] (1,0) circle (1.000);
\draw[-stealth, col0, line width=1.2pt] (-1,0) -- (0,0);
\draw[col0, line width=1.2pt] (0,0) -- (1,0);
\node[anchor=north, inner sep=4pt, font=\scriptsize] at (-0.5,0) {${\bf u}_{ij}$};
\draw[-stealth, col0, line width=1.2pt] (0,0) -- (0,1);
\node[anchor=south, inner sep=2pt, font=\scriptsize] at (0,1.1) {${\bf u}_{ij}^{\perp}$};
\draw[-stealth, black!80, line width=0.8pt]
  (-1,0) -- (-0.6,0.4)
  node[midway, above left, inner sep=1pt, font=\scriptsize]
  {$\delta{\bf c}_i$};
\draw[-stealth, black!80, line width=0.8pt]
  (1,0) -- (1.25,0.7)
  node[midway, above left, inner sep=1pt, font=\scriptsize]
  {$\delta{\bf c}_j$};
\filldraw[black] (-1,0) circle (0.038);
\filldraw[black] (1,0) circle (0.038);
\filldraw[fill=red, draw=black, line width=0.4pt] (0,0) circle (0.038);
\node[anchor=north east, inner sep=4pt, font=\scriptsize] at (-1,0) {${\bf c}_i$};
\node[anchor=north east, inner sep=4pt, font=\scriptsize] at (1,0) {${\bf c}_j$};
\end{tikzpicture}
\caption{Geometry at a realised contact, showing the normal direction
\({\bf u}_{ij}\), the tangential direction \({\bf u}_{ij}^{\perp}\), and
perturbations of the centres. The first order non overlap condition is
\(\langle {\bf u}_{ij},\delta{\bf c}_j-\delta{\bf c}_i\rangle\ge 0\), with
equality corresponding to contact preservation.}
\label{fig:contact}
\end{figure}

\subsection{First order admissibility and the rolling space}
\label{subsec:admissible-rolling}

Let \({\bf c}\in\mathcal D_n\). The first order admissible cone at
\({\bf c}\) is
\[
\Adm({\bf c})
=
\left\{
\delta{\bf c}\in(\R^2)^n:
\langle {\bf u}_{ij},\delta{\bf c}_j-\delta{\bf c}_i\rangle\ge0
\ \text{for all } \{i,j\}\in E({\bf c})
\right\},
\]
where \({\bf u}_{ij}\) is defined by
\eqref{eq:first-order-nonoverlap}. This cone records the full first order
feasibility relation coming from the non overlap inequalities defining
\(\mathcal D_n\). A direction in \(\Adm({\bf c})\) may either preserve
or open a contact to first order. 

Suppose that \({\bf c}\in\mathcal C(G)\). Since
\(E({\bf c})=E(G)\), one has
\[
{\bf u}_{ij}
=
\frac{{\bf c}_j-{\bf c}_i}{2}
\qquad
\text{for all } \{i,j\}\in E(G).
\]

The rolling space records the contact preserving equality part of \(\mathsf{Adm}({\bf c})\). It is the linear subspace obtained by replacing the
inequalities of the form \eqref{eq:first-order-nonoverlap} at the active
contacts by equalities.

\begin{definition}
The rolling space at \({\bf c}\) is
\[
\Roll({\bf c})=
\left\{
\delta{\bf c}\in(\R^2)^n:
\langle {\bf u}_{ij},\delta{\bf c}_j-\delta{\bf c}_i\rangle=0
\ \text{for all } \{i,j\}\in E(G)
\right\}.
\]
\end{definition}

Thus
\[
\Roll({\bf c})\subset \Adm({\bf c}).
\]
The inclusion is generally strict, first order admissible directions may open
one or more contacts, whereas rolling directions preserve all contacts of the
fixed contact class to first order.

\subsection{The contact operator}
The contact operator plays the role of the rigidity matrix for the realised contact geometry of a hard disk configuration. The crucial difference is that it does not define the contact geometry, it records it in coordinates. The terminology is chosen to emphasise
a distinction from bar--joint rigidity. In the latter, the graph is prescribed
in advance and its edges impose fixed length constraints, whereas here the
active edges are exactly the tangencies realised by the current configuration,
while non edges remain strict separation inequalities. For a realised
configuration ${\bf c}\in\mathcal C(G)$, the contact operator is
\[
A({\bf c}) : (\R^2)^n \to \R^{E(G)}.
\]
It is the linear map defined by
\[
\bigl(A({\bf c})\,\delta{\bf c}\bigr)_{ij}
=
\langle {\bf u}_{ij},\delta{\bf c}_j-\delta{\bf c}_i\rangle
\qquad
\text{for } \{i,j\}\in E(G).
\]
Here $E(G)$ is the fixed edge set of the graph $G$ defining the class under
discussion. In particular,
\[
\Roll({\bf c})=\ker A({\bf c}).
\]
Regularity of ${\bf c}$ for $G$ is equivalent to $A({\bf c})$ having full row
rank.
\begin{lemma}
\label{lem:tangent-identification}
Let ${\bf c}\in \mathcal C(G)$ be regular for $G$. Then
\[
T_{\bf c}M(G)=\Roll({\bf c}).
\]
\end{lemma}
\begin{proof}
By Lemma~\ref{lem:manifold-rolling},
\[
T_{\bf c}M(G)=\ker DF({\bf c}).
\]
For any variation $\delta{\bf c}$ and any active edge $\{i,j\}$,
\[
Dg_{ij}({\bf c})[\delta{\bf c}]
=
2\langle {\bf c}_j-{\bf c}_i,\delta{\bf c}_j-\delta{\bf c}_i\rangle
=
4\langle {\bf u}_{ij},\delta{\bf c}_j-\delta{\bf c}_i\rangle.
\]
Thus
\[
\ker DF({\bf c})=\ker A({\bf c})=\Roll({\bf c}).
\]
\end{proof}

\subsection{Interior points and rigid motions}
The identification \(T_{\bf c}M(G)=\Roll({\bf c})\) becomes the tangent space of the fixed hull and contact class when no new contact or hull change occurs
nearby.

\begin{remark}[Relative interior]
\label{rem:interior-class}
Let $X=\mathcal F\cap\mathcal C(G)$ be a fixed hull and contact class, and let
${\bf c}\in X$ be regular for $G$. Suppose that, along $M(G)$ near ${\bf c}$,
no non edge comes into contact. Thus there exists a neighbourhood $U_0$ of
${\bf c}$ in $(\R^2)^n$ such that
\[
\|{\bf d}_i-{\bf d}_j\|>2
\qquad
\text{for all } {\bf d}\in U_0\cap M(G)
\text{ and all } \{i,j\}\notin E(G).
\]
After shrinking $U_0$ if necessary, assume also that
\[
U_0\cap\mathcal D_n\subset\mathcal F.
\]
Then
\[
X\cap U_0=M(G)\cap U_0.
\]
In particular, ${\bf c}$ is an interior point of $X$ relative to $M(G)$, and by
Lemma~\ref{lem:tangent-identification},
\[
T_{\bf c}X=T_{\bf c}M(G)=\Roll({\bf c}).
\]
\end{remark}

For a fixed hull class $\mathcal F$ with cyclic hull edge list $\mathcal B$, set
\[
\widetilde\Per({\bf d})
=
\sum_{(i,j)\in\mathcal B}\|{\bf d}_j-{\bf d}_i\|+2\pi.
\]
Then $\widetilde\Per=\Per$ on $\mathcal F$, and hence on $X$.

\begin{proposition}[First order class criticality]
\label{prop:first order-criticality}
Let ${\bf c}$ be a relative interior regular point of a fixed
hull and contact class
\[
X=\mathcal F\cap\mathcal C(G).
\]
Then the following are equivalent:

\begin{itemize}
    \item ${\bf c}$ is class critical, that is,
\[
D\Per({\bf c})[\delta{\bf c}]=0
\qquad
\text{for every }\delta{\bf c}\in\Roll({\bf c});
\]
\item
\[
\nabla\widetilde\Per({\bf c})\perp\Roll({\bf c});
\]
\item there exists a multiplier vector
\[
\lambda=(\lambda_{ij})_{\{i,j\}\in E(G)}
\]
such that
\[
\nabla\widetilde\Per({\bf c})
=
A({\bf c})^\top\lambda .
\]
Equivalently, using the normalised contact constraints
\[
h_{ij}({\bf d})=\|{\bf d}_j-{\bf d}_i\|-2,
\]
one has
\[
\nabla\widetilde\Per({\bf c})
=
\sum_{\{i,j\}\in E(G)}
\lambda_{ij}\nabla h_{ij}({\bf c}).
\]
\end{itemize}
\end{proposition}

\begin{proof}
Since ${\bf c}$ is a relative interior regular point of \(X\),
Remark~\ref{rem:interior-class} gives
\[
T_{\bf c}X=\Roll({\bf c})=\ker A({\bf c}).
\]
On a hull stable neighbourhood, \(\widetilde\Per=\Per\) on \(X\). Hence
\({\bf c}\) is class critical exactly when
\[
D\widetilde\Per({\bf c})[\delta{\bf c}]=0
\qquad
\text{for every }\delta{\bf c}\in\Roll({\bf c}).
\]

With respect to the standard inner product on \((\R^2)^n\), this is equivalent
to
\[
\nabla\widetilde\Per({\bf c})\perp\Roll({\bf c}).
\]
Since
\[
\Roll({\bf c})=\ker A({\bf c}),
\]
the finite-dimensional identity
\[
(\ker A({\bf c}))^\perp=\operatorname{im} A({\bf c})^\top
\]
shows that this is equivalent to the existence of
\(\lambda\in\R^{E(G)}\) such that
\[
\nabla\widetilde\Per({\bf c})=A({\bf c})^\top\lambda.
\]
Finally, the row of \(A({\bf c})\) indexed by \(\{i,j\}\) is precisely the
gradient of the normalised contact constraint
\[
h_{ij}({\bf d})=\|{\bf d}_j-{\bf d}_i\|-2
\]
at \({\bf c}\). Therefore
\[
A({\bf c})^\top\lambda
=
\sum_{\{i,j\}\in E(G)}
\lambda_{ij}\nabla h_{ij}({\bf c}),
\]
which gives the stated equivalent formulation.
\end{proof}

Since \(\Per\) is invariant under rigid motions, we quotient out the standard
infinitesimal rigid motion space
\[
R({\bf c})
=
\left\{
({\bf a}+\omega J{\bf c}_1,\dots,{\bf a}+\omega J{\bf c}_n):
{\bf a}\in\R^2,\ \omega\in\R
\right\},
\]
where \(J(x,y)=(-y,x)\).

\begin{lemma}
\label{lem:R-in-Roll}
For every ${\bf c}\in\mathcal C(G)$ one has $R({\bf c})\subset\Roll({\bf c})$.
\end{lemma}

\begin{proof}
A pure translation $\delta{\bf c}_i={\bf a}$ for all $i$ satisfies
$
\delta{\bf c}_j-\delta{\bf c}_i={\bf 0},
$
hence all rolling constraints are trivially satisfied. For the infinitesimal
rotation $\delta{\bf c}_i=\omega J{\bf c}_i$, one has
\[
\delta{\bf c}_j-\delta{\bf c}_i
=
\omega J({\bf c}_j-{\bf c}_i).
\]
Since ${\bf u}_{ij}=({\bf c}_j-{\bf c}_i)/2$, it follows that
\[
\langle {\bf u}_{ij},\omega J({\bf c}_j-{\bf c}_i)\rangle
=
\frac{\omega}{2}\langle {\bf c}_j-{\bf c}_i,J({\bf c}_j-{\bf c}_i)\rangle
=0,
\]
because $\langle {\bf w},J{\bf w}\rangle=0$ for every ${\bf w}\in\R^2$.
\end{proof}

\subsection{Leaving a class and incident classes}
Let \(X=\mathcal F\cap\mathcal C(G)\), and let
\({\bf c}^\ast\in X\) be a relative interior point of the class. A \(C^1\) path
\({\bf c}(t)\) in \(\mathcal D_n\) with \({\bf c}(0)={\bf c}^\ast\) either
remains in \(X\) for all sufficiently small \(t>0\), or exits through at least
one of three mechanisms: an active contact opens, a new contact forms, or the
hull data change.
\begin{enumerate}
\item {\bf Edge deletion:}
      \(\|{\bf c}_i(t)-{\bf c}_j(t)\|>2\) for some \(\{i,j\}\in E(G)\);
\item {\bf Edge addition:}
      \(\|{\bf c}_i(t)-{\bf c}_j(t)\|=2\) for some \(\{i,j\}\notin E(G)\);
\item {\bf Hull change:} the path exits \(\mathcal F\).
\end{enumerate}
These mechanisms are not mutually exclusive. 

\begin{remark}[Closure convention]
\label{rem:closure}
All closures of sets of the form $\mathcal F\cap\mathcal C(G')$ are taken in
the ambient space $(\R^2)^n$. Boundary configurations are treated through their
incident fixed classes.
\end{remark}

\begin{lemma}[One-sided reduction to incident fixed classes]
\label{lem:descent-adjacent}
Let ${\bf c}\in\mathcal D_n$. Let
\[
\mathcal A({\bf c})
=
\{X_\alpha=\mathcal F_\alpha\cap\mathcal C(G_\alpha):
{\bf c}\in\overline{X_\alpha}\}
\]
be the finite collection of fixed hull and contact classes whose closures
contain ${\bf c}$, where each class is determined by an exact realised contact
graph and a cyclic hull datum. Suppose that there exists $\rho>0$ such that
\[
\Per({\bf d})\ge \Per({\bf c})
\qquad
\text{for every }X_\alpha\in\mathcal A({\bf c})
\text{ and every }{\bf d}\in\overline{X_\alpha}\cap B_\rho({\bf c}),
\]
where \(B_\rho({\bf c})\) denotes the open ball in \((\R^2)^n\). Then
${\bf c}$ is a local minimiser of $\Per$ in $\mathcal D_n$.
\end{lemma}

\begin{proof}
Assume that ${\bf c}$ is not a local minimiser. Then there exists a sequence
${\bf c}^k\to{\bf c}$ in $\mathcal D_n$ such that
\[
\Per({\bf c}^k)<\Per({\bf c})
\qquad
\text{for all }k.
\]

There are only finitely many contact graphs on $\{1,\ldots,n\}$ and only
finitely many cyclic hull data on this labelled vertex set. Hence, after
passing to a subsequence, there is a fixed hull and contact class
\(X_\alpha\) such that
\[
{\bf c}^k\in\overline{X_\alpha}
\qquad
\text{for all }k.
\]
Since ${\bf c}^k\to{\bf c}$, it follows that ${\bf c}\in\overline{X_\alpha}$.
Thus \(X_\alpha\in\mathcal A({\bf c})\). For \(k\) sufficiently large,
\[
{\bf c}^k\in\overline{X_\alpha}\cap B_\rho({\bf c}).
\]
The assumed one-sided inequality on the incident class gives
\[
\Per({\bf c}^k)\ge \Per({\bf c}),
\]
contradicting the strict descent assumption. Therefore no strictly descending
sequence exists, and ${\bf c}$ is a local minimiser in $\mathcal D_n$.
\end{proof}

\section{First order obstruction: hull leaves}
\label{sec:obstructions}
Throughout this section we let ${\bf c}\in X=\mathcal F\cap\mathcal C(G)$ be a relative interior regular point of a fixed hull and contact class.

\begin{proposition}[Direct first order exclusion]
\label{prop:direct-first order-exclusion}
Let ${\bf c}$ be a relative interior regular point of a fixed hull and contact
class \(X=\mathcal F\cap\mathcal C(G)\). If there exists
$\delta{\bf c}\in \Roll({\bf c})$ such that
$D\Per({\bf c})[\delta{\bf c}]<0$, then ${\bf c}$ is not class critical. In
particular, ${\bf c}$ is not a local minimiser of $\Per$ on $X$.
\end{proposition}

\begin{proof}
Class criticality means that $D\Per({\bf c})[\eta]=0$ for every $\eta\in\Roll({\bf c})$. So the existence of $\delta{\bf c}\in\Roll({\bf c})$ with $D\Per({\bf c})[\delta{\bf c}]<0$ rules out class criticality.

Since \({\bf c}\) is a relative interior regular point of \(X\), there exists a $C^2$ curve ${\bf c}(t)\subset X$ with ${\bf c}(0)={\bf c}$ and $\dot{\bf c}(0)=\delta{\bf c}$. Hence $\left.\frac{d}{dt}\right|_{t=0}\Per({\bf c}(t))<0$, so $\Per({\bf c}(t))<\Per({\bf c})$ for all sufficiently small $t>0$. Thus ${\bf c}$ is not a local minimiser of $\Per$ on $X$.
\end{proof}

\subsection{Hull leaves}
A hull leaf is a hull vertex of degree one in \(G\).

\begin{theorem}
\label{thm:leaf-interior-descent}
Let \(X=\mathcal F\cap\mathcal C(G)\), and let \({\bf c}\in X\) be a relative
interior regular point. Let \(i\) be a hull leaf, and let \(j\) be its unique
neighbour in \(G\). Let \({\bf t}_{\mathrm{in}}(i)\) and
\({\bf t}_{\mathrm{out}}(i)\) be the incoming and outgoing unit tangent vectors
of the cyclic hull at \(i\), and set
\[
{\bf g}_i
=
{\bf t}_{\mathrm{in}}(i)-{\bf t}_{\mathrm{out}}(i).
\]
If \({\bf c}\) is class critical, then \({\bf g}_i\) is parallel to the contact
normal \({\bf u}_{ij}\). Consequently, if \({\bf g}_i\) is not parallel to
\({\bf u}_{ij}\), then \({\bf c}\) is not class critical and in particular
cannot be a local minimiser of \(\Per\) on \(X\).
\end{theorem}

\begin{proof}
Choose $\delta{\bf c}$ supported at vertex $i$ and orthogonal to ${\bf u}_{ij}$.
Since $i$ has exactly one active neighbour, the only rolling constraint
involving $i$ is
\[
\langle {\bf u}_{ij},\delta{\bf c}_i\rangle=0
\]
with $\delta{\bf c}_j=0$, which is satisfied. Hence
\[
\delta{\bf c}\in\Roll({\bf c}).
\]
Since ${\bf c}$ is a relative interior regular point of $X$, there exists a
$C^2$ curve in $X$ tangent to $\delta{\bf c}$. The first variation gives
\[
D\Per({\bf c})[\delta{\bf c}]
=
\langle {\bf g}_i,\delta{\bf c}_i\rangle.
\]
If ${\bf g}_i$ has a nonzero component orthogonal to ${\bf u}_{ij}$, choosing
the sign of $\delta{\bf c}_i$ yields a negative first variation, contradicting
class criticality. Therefore
\[
{\bf g}_i\in\operatorname{span}\{{\bf u}_{ij}\}.
\]
\end{proof}

\section{Second variation on a fixed hull contact class}
\label{sec:second-variation}

This section studies the intrinsic second variation on a fixed hull and contact class. The restricted Hessian separates descent directions from rigid motions and class flexes, detecting instability, flat degeneracy, and rigidity. Appendix~\ref{sec:examples-four-disks} exemplifies the theory for four disk clusters: the square, the rhombus family, the rhombus with diagonal contact, and the star with one central disk and three leaves. The intrinsic spectra distinguish the three outcomes.

\subsection{Local extension and contact constraints}
On the fixed hull class, use the smooth extension
\[
\widetilde\Per({\bf d})
=
\sum_{(p,q)\in\mathcal B}\|{\bf d}_q-{\bf d}_p\|+2\pi,
\qquad
\widetilde\Per=\Per \text{ on } X.
\]

For the second variation we use the normalised contact constraints
\[
h_{ij}({\bf d})=\|{\bf d}_j-{\bf d}_i\|-2,
\qquad
\{i,j\}\in E(G).
\]
They have the same zero set as the squared constraints \(g_{ij}\), but their
first derivatives are exactly the contact operator rows:
\[
Dh_{ij}({\bf c})[\delta{\bf c}]
=
\langle {\bf u}_{ij},\delta{\bf c}_j-\delta{\bf c}_i\rangle.
\]
Thus class criticality gives the multiplier equation
\[
\nabla\widetilde\Per({\bf c})
=
\sum_{\{i,j\}\in E(G)}\lambda_{ij}\nabla h_{ij}({\bf c}).
\]

\begin{theorem}
\label{thm:second-variation-per}
Let ${\bf c}$ be a relative interior regular class critical point of a fixed hull and
contact class $X=\mathcal F\cap\mathcal C(G)$. Let $\widetilde\Per$ be the
local smooth extension of $\Per$ obtained by fixing the hull cycle, and let
\(\lambda=(\lambda_{ij})_{\{i,j\}\in E(G)}\) be the multiplier vector satisfying
\[
\nabla\widetilde\Per({\bf c})
=
\sum_{\{i,j\}\in E(G)}
\lambda_{ij}\nabla h_{ij}({\bf c}),
\qquad
h_{ij}({\bf d})=\|{\bf d}_j-{\bf d}_i\|-2.
\]
Then for every
\(\delta{\bf c}\in T_{\bf c}X=\Roll({\bf c})\) and every \(C^2\) curve
${\bf c}(t)\subset X$ with ${\bf c}(0)={\bf c}$ and
\(\dot{\bf c}(0)=\delta{\bf c}\), one has

\begin{equation}
\label{eq:lagrangian-second-variation}
\left.\frac{d^2}{dt^2}\right|_{t=0}\Per({\bf c}(t))
=
D^2\widetilde\Per({\bf c})[\delta{\bf c},\delta{\bf c}]
-
\sum_{\{i,j\}\in E(G)}
\lambda_{ij}D^2 h_{ij}({\bf c})[\delta{\bf c},\delta{\bf c}].
\end{equation}

Moreover, such a \(C^2\) curve exists for every
\(\delta{\bf c}\in\Roll({\bf c})\).
\end{theorem}

\begin{proof}
Since \({\bf c}\) is a relative interior regular point of \(X\), the manifold
\(M(G)\) is smooth near \({\bf c}\), and \(X=M(G)\) in a neighbourhood of
\({\bf c}\) inside \(M(G)\). Hence every
$\delta{\bf c}\in T_{\bf c}X=T_{\bf c}M(G)$ is tangent to a $C^2$ curve
${\bf c}(t)\subset X$ with ${\bf c}(0)={\bf c}$ and
$\dot{\bf c}(0)=\delta{\bf c}$.

Fix such a curve and write $a=\ddot{\bf c}(0)$. Since $\Per=\widetilde\Per$ on
$X$,
\[
\left.\frac{d^2}{dt^2}\right|_{t=0}\Per({\bf c}(t))
=
D^2\widetilde\Per({\bf c})[\delta{\bf c},\delta{\bf c}]
+
\langle \nabla \widetilde\Per({\bf c}),a\rangle.
\]
Because ${\bf c}(t)\subset M(G)$, one has $h_{ij}({\bf c}(t))=0$ for every
$\{i,j\}\in E(G)$. Differentiating twice at $t=0$ gives
\[
0
=
D^2 h_{ij}({\bf c})[\delta{\bf c},\delta{\bf c}]
+
\langle \nabla h_{ij}({\bf c}),a\rangle.
\]
Multiplying by $\lambda_{ij}$ and summing over $E(G)$, then using
\[
\nabla \widetilde\Per({\bf c})
=
\sum_{\{i,j\}\in E(G)}\lambda_{ij}\,\nabla h_{ij}({\bf c}),
\]
yields
\[
\langle \nabla \widetilde\Per({\bf c}),a\rangle
=
-
\sum_{\{i,j\}\in E(G)}\lambda_{ij}\,D^2 h_{ij}({\bf c})[\delta{\bf c},\delta{\bf c}].
\]

Substituting proves \eqref{eq:lagrangian-second-variation}. The right-hand side
depends only on $\delta{\bf c}$, so the second derivative is independent of the
chosen admissible curve.
\end{proof}

\begin{definition}
For $\delta{\bf c}\in\Roll({\bf c})$ define
\[
D^2\Per({\bf c})[\delta{\bf c}]
=
\left.\frac{d^2}{dt^2}\right|_{t=0}\Per({\bf c}(t)),
\]
where ${\bf c}(t)\subset X$ is any $C^2$ curve with
${\bf c}(0)={\bf c}$ and $\dot{\bf c}(0)=\delta{\bf c}$. This is well defined by
Theorem~\ref{thm:second-variation-per}.
\end{definition}

\subsection{Hull contribution}
For each oriented hull edge \((p,q)\in\mathcal B\), set
\[
{\bf t}_{pq}
=
\frac{{\bf c}_q-{\bf c}_p}{\|{\bf c}_q-{\bf c}_p\|}.
\]
The hull contribution comes from differentiating the edge length
\[
{\bf d}\longmapsto \|{\bf d}_q-{\bf d}_p\|.
\]
Its second derivative at \({\bf c}\), in the relative displacement
\(\delta{\bf c}_q-\delta{\bf c}_p\), is
\[
\frac{1}{\|{\bf c}_q-{\bf c}_p\|}
\left\langle
\left(I-{\bf t}_{pq}{\bf t}_{pq}^\top\right)(\delta{\bf c}_q-\delta{\bf c}_p),
\delta{\bf c}_q-\delta{\bf c}_p
\right\rangle.
\]
Thus each hull edge contributes the \(2\times2\) transverse Hessian block
\begin{equation}
\label{eq:euclidean-block}
M_{pq}
=
\frac{1}{\|{\bf c}_q-{\bf c}_p\|}
\left(I-{\bf t}_{pq}{\bf t}_{pq}^\top\right).
\end{equation}
Summing the edge contributions gives
\[
D^2\widetilde\Per({\bf c})[\delta{\bf c},\delta{\bf c}]
=
\sum_{(p,q)\in\mathcal B}
\left\langle
M_{pq}(\delta{\bf c}_q-\delta{\bf c}_p),\delta{\bf c}_q-\delta{\bf c}_p
\right\rangle.
\]
The hull Hessian \(\mathcal H_{\mathrm{hull}}\) is the symmetric block matrix
obtained by assembling these \(M_{pq}\) blocks in the \(p,q\) coordinate blocks,
so that
\[
D^2\widetilde\Per({\bf c})[\delta{\bf c},\delta{\bf c}]
=
\delta{\bf c}^\top\mathcal H_{\mathrm{hull}}\delta{\bf c}.
\]

\begin{remark}
Each \(M_{pq}\) is positive semidefinite, since
\(I-{\bf t}_{pq}{\bf t}_{pq}^\top\) is the orthogonal projection onto
\({\bf t}_{pq}^{\perp}\). Hence
\[
D^2\widetilde\Per({\bf c})[\delta{\bf c},\delta{\bf c}]\ge 0
\qquad
\text{for all }\delta{\bf c}\in(\R^2)^n.
\]
Thus \(\widetilde\Per\) is convex on the hull stable neighbourhood.
\end{remark}

\subsection{Contact curvature contribution}
For each active contact \(\{i,j\}\in E(G)\), the Hessian of \(h_{ij}\) in the
relative displacement \(\delta{\bf c}_j-\delta{\bf c}_i\) is
\[
\frac{1}{\|{\bf c}_j-{\bf c}_i\|}
\left(I-{\bf u}_{ij}{\bf u}_{ij}^{\top}\right)
=
\frac{1}{2}
\left(I-{\bf u}_{ij}{\bf u}_{ij}^{\top}\right).
\]
Thus the Lagrangian term \(-\lambda_{ij}D^2h_{ij}({\bf c})\) contributes the
\(2\times2\) transverse block
\begin{equation}
\label{eq:geometric-block}
K_{ij}
=
-\frac{\lambda_{ij}}{2}
\left(I-{\bf u}_{ij}{\bf u}_{ij}^{\top}\right).
\end{equation}
The contact Hessian \(\mathcal H_{\mathrm{cont}}\) is the symmetric block matrix
obtained by assembling these \(K_{ij}\) blocks in the \(i,j\) coordinate blocks,
so that
\[
-\sum_{\{i,j\}\in E(G)}\lambda_{ij}D^2 h_{ij}({\bf c})[\delta{\bf c},\delta{\bf c}]
=
\delta{\bf c}^{\top}\mathcal H_{\mathrm{cont}}\delta{\bf c}.
\]

\subsection{The total and intrinsic Hessians}
\begin{definition}
The total Hessian is
\[
\mathcal H
=
\mathcal H_{\mathrm{hull}}+\mathcal H_{\mathrm{cont}}.
\]
\end{definition}

By Theorem~\ref{thm:second-variation-per},
\begin{equation}
\label{eq:second-var-quadratic}
D^2\Per({\bf c})[\delta{\bf c}]
=
\delta{\bf c}^\top\mathcal H\delta{\bf c}
\qquad
\text{for all }\delta{\bf c}\in\Roll({\bf c}).
\end{equation}

\begin{definition}
Let $Z$ be a matrix whose columns form an orthonormal basis of
\[
\Roll({\bf c})\cap R({\bf c})^\perp.
\]
The intrinsic Hessian of $\Per$ at ${\bf c}$ is
\begin{equation}
\label{eq:intrinsic-hessian}
\mathcal H_{\mathrm{intr}}
=
Z^\top\mathcal HZ.
\end{equation}
\end{definition}

If $\delta{\bf c}=Zy\in\Roll({\bf c})\cap R({\bf c})^\perp$, then
\[
D^2\Per({\bf c})[\delta{\bf c}]=y^\top\mathcal H_{\mathrm{intr}}y.
\]

\subsection{Reduced second order diagnostics}
The intrinsic Hessian is used here as a diagnostic on the reduced rolling space
\[
\mathcal R_{\bf c}:=\Roll({\bf c})\cap R({\bf c})^\perp.
\]
In the finite classes considered below, the relevant outcomes are second order instability, rigidity modulo rigid motions, and flat degeneracy.

\begin{definition}[Reduced variational type]
\label{def:reduced-type}
Let ${\bf c}$ be a relative interior regular class-critical point of a fixed hull and contact class \(X=\mathcal F\cap\mathcal C(G)\). We say that \({\bf c}\) is:
\begin{itemize}
\item {\bf second order unstable} if there exists \(0\ne\delta{\bf c}\in\mathcal R_{\bf c}\) such that
\[
D^2\Per({\bf c})[\delta{\bf c}]<0;
\]
\item {\bf rigid modulo rigid motions} if
\[
\mathcal R_{\bf c}=\{0\};
\]
\item {\bf flat degenerate} if
\[
\mathcal R_{\bf c}\ne\{0\},
\]
\[
D^2\Per({\bf c})[\delta{\bf c}]\ge 0
\qquad
\text{for all }\delta{\bf c}\in\mathcal R_{\bf c},
\]
and the intrinsic Hessian has nontrivial kernel on \(\mathcal R_{\bf c}\).
\end{itemize}
\end{definition}

\begin{proposition}[Second order exclusion]
\label{prop:second order-exclusion}
Let ${\bf c}$ be a relative interior regular class-critical point of a fixed hull and contact class \(X=\mathcal F\cap\mathcal C(G)\). If \({\bf c}\) is second order unstable, then \({\bf c}\) is not a local minimiser of \(\Per\) on \(X\), modulo rigid motions.
\end{proposition}

\begin{proof}
Let \(0\ne\delta{\bf c}\in\mathcal R_{\bf c}\) satisfy
\[
D^2\Per({\bf c})[\delta{\bf c}]<0.
\]
Since \({\bf c}\) is a relative interior regular point of \(X\), there is a \(C^2\) curve
\[
{\bf c}(t)\subset X,\qquad {\bf c}(0)={\bf c},\qquad
\dot{\bf c}(0)=\delta{\bf c}.
\]
Class criticality gives
\[
\left.\frac{d}{dt}\right|_{t=0}\Per({\bf c}(t))=0,
\]
while the choice of \(\delta{\bf c}\) gives
\[
\left.\frac{d^2}{dt^2}\right|_{t=0}\Per({\bf c}(t))<0.
\]
Hence \(\Per({\bf c}(t))<\Per({\bf c})\) for all sufficiently small nonzero \(t\) of one sign. Therefore \({\bf c}\) is not a local minimiser of \(\Per\) on \(X\), modulo rigid motions.
\end{proof}

\begin{proposition}[Rigidity modulo rigid motions]
\label{prop:rigid-reduced-class}
Let ${\bf c}$ be a relative interior regular point of a fixed hull and contact class \(X=\mathcal F\cap\mathcal C(G)\). If
\[
\mathcal R_{\bf c}=\{0\},
\]
then \({\bf c}\) has no nontrivial first order admissible motion in \(X\) after quotienting by rigid motions. Equivalently,
\[
\Roll({\bf c})=R({\bf c}).
\]
\end{proposition}

\begin{proof}
By Lemma~\ref{lem:R-in-Roll}, \(R({\bf c})\subset\Roll({\bf c})\). Let \(v\in\Roll({\bf c})\), and decompose
\[
v=r+w,
\qquad
r\in R({\bf c}),\quad w\in R({\bf c})^\perp.
\]
Since \(r\in R({\bf c})\subset\Roll({\bf c})\) and \(\Roll({\bf c})\) is linear, we have \(w=v-r\in\Roll({\bf c})\). Hence
\[
w\in\Roll({\bf c})\cap R({\bf c})^\perp=\{0\}.
\]
Thus \(v=r\in R({\bf c})\), so \(\Roll({\bf c})\subset R({\bf c})\). The reverse inclusion is Lemma~\ref{lem:R-in-Roll}; therefore
\[
\Roll({\bf c})=R({\bf c}).
\]
\end{proof}

\begin{proposition}[Flat class]
\label{prop:flat-class}
Let \({\bf c}\) be a relative interior regular point of a fixed hull and
contact class
\[
X=\mathcal F\cap\mathcal C(G),
\]
and suppose that
\[
\dim\mathcal R_{\bf c}=k>0.
\]
If every hull edge is an active contact edge, then, modulo rigid motions,
\(X\) is locally a \(k\) dimensional manifold near \({\bf c}\), and
\(\Per\) is locally constant on \(X\). Consequently, modulo rigid motions,
the nearby configurations form a local \(k\) parameter perimeter preserving
family and are therefore nonisolated.
\end{proposition}

\begin{proof}
Since \({\bf c}\) is a relative interior regular point of \(X\), the class is
locally a smooth manifold with
\[
T_{\bf c}X=\Roll({\bf c}).
\]
After quotienting by rigid motions, the reduced tangent space at
\({\bf c}\) is
\[
\mathcal R_{\bf c}
=
\Roll({\bf c})\cap R({\bf c})^\perp.
\]
Hence the hypothesis
\[
\dim\mathcal R_{\bf c}=k
\]
implies that, modulo rigid motions, \(X\) is locally a \(k\) dimensional
manifold near \({\bf c}\).

Since the hull data are fixed on \(X\), the centre hull perimeter is the sum
of the lengths of its hull edges. By hypothesis every hull edge is an active
contact edge, and every configuration in \(X\) has contact graph \(G\).
Hence every hull edge has constant length \(2\) throughout \(X\). Therefore
\[
\Per({\bf d})=\Per({\bf c})
\]
for every \({\bf d}\in X\) sufficiently near \({\bf c}\). Thus \(\Per\) is
locally constant on \(X\). Consequently, modulo rigid motions, the nearby
configurations form a local \(k\) parameter perimeter preserving family and
are therefore nonisolated.
\end{proof}

\begin{figure}[ht]
\centering
\setlength{\tabcolsep}{8pt}
\renewcommand{\arraystretch}{1.5}
\tikzset{diskhull/.style={black!100,line width=.5pt}}
\begin{tabular}{cccc}

\begin{tikzpicture}[scale=\graphscale, line cap=round, line join=round]
\node[anchor=north west, inner sep=1pt, text=black!65, font=\scriptsize] at (-3.2,2.9) {(1)};
\draw[col0,line width=1.2pt] (0.000,0.000) -- (0.000,2.000);
\draw[col1,line width=1.2pt] (0.000,0.000) -- (-2.000,0.000);
\draw[col2,line width=1.2pt] (0.000,0.000) -- (2.000,0.000);
\draw[col3,line width=1.2pt] (0.000,0.000) -- (0.000,-2.000);
\fill[black!17,opacity=0.45] (0.000,0.000) circle (1.000);
\fill[black!17,opacity=0.45] (0.000,2.000) circle (1.000);
\fill[black!17,opacity=0.45] (-2.000,0.000) circle (1.000);
\fill[black!17,opacity=0.45] (2.000,0.000) circle (1.000);
\fill[black!17,opacity=0.45] (0.000,-2.000) circle (1.000);
\draw[gray,line width=0.8pt] (0.000,0.000) circle (1.000);
\draw[gray,line width=0.8pt] (0.000,2.000) circle (1.000);
\draw[gray,line width=0.8pt] (-2.000,0.000) circle (1.000);
\draw[gray,line width=0.8pt] (2.000,0.000) circle (1.000);
\draw[gray,line width=0.8pt] (0.000,-2.000) circle (1.000);
\draw[diskhull]
  (-2.707,-0.707)
  -- (-0.707,-2.707)
  arc[start angle=-135.0,end angle=-45.0,radius=1.000]
  -- (2.707,-0.707)
  arc[start angle=-45.0,end angle=45.0,radius=1.000]
  -- (0.707,2.707)
  arc[start angle=45.0,end angle=135.0,radius=1.000]
  -- (-2.707,0.707)
  arc[start angle=135.0,end angle=225.0,radius=1.000];
\filldraw[black] (0.000,0.000) circle (0.05);
\filldraw[black] (0.000,2.000) circle (0.05);
\filldraw[black] (-2.000,0.000) circle (0.05);
\filldraw[black] (2.000,0.000) circle (0.05);
\filldraw[black] (0.000,-2.000) circle (0.05);
\end{tikzpicture}
&
\begin{tikzpicture}[scale=\graphscale, line cap=round, line join=round] 
\node[anchor=north west, inner sep=1pt, text=black!65, font=\scriptsize] at (-4.6,3.5) {(2)};
\draw[col0,line width=1.2pt] (-3.366,0.000) -- (-1.366,0.000);
\draw[col1,line width=1.2pt] (-1.366,0.000) -- (0.634,0.000);
\draw[col2,line width=1.2pt] (0.634,0.000) -- (2.049,1.414);
\draw[col3,line width=1.2pt] (0.634,0.000) -- (2.049,-1.414);
\fill[black!17,opacity=0.45] (-3.366,0.000) circle (1.000);
\fill[black!17,opacity=0.45] (-1.366,0.000) circle (1.000);
\fill[black!17,opacity=0.45] (0.634,0.000) circle (1.000);
\fill[black!17,opacity=0.45] (2.049,1.414) circle (1.000);
\fill[black!17,opacity=0.45] (2.049,-1.414) circle (1.000);
\draw[gray,line width=0.8pt] (-3.366,0.000) circle (1.000);
\draw[gray,line width=0.8pt] (-1.366,0.000) circle (1.000);
\draw[gray,line width=0.8pt] (0.634,0.000) circle (1.000);
\draw[gray,line width=0.8pt] (2.049,1.414) circle (1.000);
\draw[gray,line width=0.8pt] (2.049,-1.414) circle (1.000);
\draw[diskhull]
  (-3.619,-0.968)
  -- (1.796,-2.382)
  arc[start angle=-104.6,end angle=0.0,radius=1.000]
  -- (3.049,1.414)
  arc[start angle=0.0,end angle=104.6,radius=1.000]
  -- (-3.619,0.968)
  arc[start angle=104.6,end angle=255.4,radius=1.000];
\filldraw[black] (-3.366,0.000) circle (0.05);
\filldraw[black] (-1.366,0.000) circle (0.05);
\filldraw[black] (0.634,0.000) circle (0.05);
\filldraw[black] (2.049,1.414) circle (0.05);
\filldraw[black] (2.049,-1.414) circle (0.05);
\end{tikzpicture}
&
\begin{tikzpicture}[scale=\graphscale, line cap=round, line join=round] 
\node[anchor=north west, inner sep=1pt, text=black!65, font=\scriptsize] at (-3.2,3.4) {(3)};
\draw[col0,line width=1.2pt] (0.000,-0.127) -- (1.414,-1.541);
\draw[col1,line width=1.2pt] (0.000,-0.127) -- (-1.414,-1.541);
\draw[col2,line width=1.2pt] (0.000,-0.127) -- (1.000,1.605);
\draw[col3,line width=1.2pt] (0.000,-0.127) -- (-1.000,1.605);
\draw[col4,line width=1.2pt] (1.000,1.605) -- (-1.000,1.605);
\fill[black!17,opacity=0.45] (0.000,-0.127) circle (1.000);
\fill[black!17,opacity=0.45] (1.414,-1.541) circle (1.000);
\fill[black!17,opacity=0.45] (-1.414,-1.541) circle (1.000);
\fill[black!17,opacity=0.45] (1.000,1.605) circle (1.000);
\fill[black!17,opacity=0.45] (-1.000,1.605) circle (1.000);
\draw[gray,line width=0.8pt] (0.000,-0.127) circle (1.000);
\draw[gray,line width=0.8pt] (1.414,-1.541) circle (1.000);
\draw[gray,line width=0.8pt] (-1.414,-1.541) circle (1.000);
\draw[gray,line width=0.8pt] (1.000,1.605) circle (1.000);
\draw[gray,line width=0.8pt] (-1.000,1.605) circle (1.000);
\draw[diskhull]
  (-1.414,-2.541)
  -- (1.414,-2.541)
  arc[start angle=-90.0,end angle=7.5,radius=1.000]
  -- (1.991,1.735)
  arc[start angle=7.5,end angle=90.0,radius=1.000]
  -- (-1.000,2.605)
  arc[start angle=90.0,end angle=172.5,radius=1.000]
  -- (-2.405,-1.411)
  arc[start angle=172.5,end angle=270.0,radius=1.000];
\filldraw[black] (0.000,-0.127) circle (0.05);
\filldraw[black] (1.414,-1.541) circle (0.05);
\filldraw[black] (-1.414,-1.541) circle (0.05);
\filldraw[black] (1.000,1.605) circle (0.05);
\filldraw[black] (-1.000,1.605) circle (0.05);
\end{tikzpicture}
&
\begin{tikzpicture}[scale=\graphscale, line cap=round, line join=round]
\node[anchor=north west, inner sep=1pt, text=black!65, font=\scriptsize] at (-4.0,3.4) {(4)};
\draw[col0,line width=1.2pt] (-2.731,0.000) -- (-0.731,0.000);
\draw[col1,line width=1.2pt] (-0.731,0.000) -- (0.683,1.414);
\draw[col2,line width=1.2pt] (-0.731,0.000) -- (0.683,-1.414);
\draw[col3,line width=1.2pt] (0.683,1.414) -- (2.097,0.000);
\draw[col4,line width=1.2pt] (0.683,-1.414) -- (2.097,0.000);
\fill[black!17,opacity=0.45] (-2.731,0.000) circle (1.000);
\fill[black!17,opacity=0.45] (-0.731,0.000) circle (1.000);
\fill[black!17,opacity=0.45] (0.683,1.414) circle (1.000);
\fill[black!17,opacity=0.45] (0.683,-1.414) circle (1.000);
\fill[black!17,opacity=0.45] (2.097,0.000) circle (1.000);
\draw[gray,line width=0.8pt] (-2.731,0.000) circle (1.000);
\draw[gray,line width=0.8pt] (-0.731,0.000) circle (1.000);
\draw[gray,line width=0.8pt] (0.683,1.414) circle (1.000);
\draw[gray,line width=0.8pt] (0.683,-1.414) circle (1.000);
\draw[gray,line width=0.8pt] (2.097,0.000) circle (1.000);
\draw[diskhull]
  (-3.114,-0.924)
  -- (0.300,-2.338)
  arc[start angle=-112.5,end angle=-45.0,radius=1.000]
  -- (2.804,-0.707)
  arc[start angle=-45.0,end angle=45.0,radius=1.000]
  -- (1.390,2.121)
  arc[start angle=45.0,end angle=112.5,radius=1.000]
  -- (-3.114,0.924)
  arc[start angle=112.5,end angle=247.5,radius=1.000];
\filldraw[black] (-2.731,0.000) circle (0.05);
\filldraw[black] (-0.731,0.000) circle (0.05);
\filldraw[black] (0.683,1.414) circle (0.05);
\filldraw[black] (0.683,-1.414) circle (0.05);
\filldraw[black] (2.097,0.000) circle (0.05);
\end{tikzpicture}
\\[0.3cm]

\begin{tikzpicture}[scale=\graphscale, line cap=round, line join=round]
\node[anchor=north west, inner sep=1pt, text=black!65, font=\scriptsize] at (-4.2,3.6) {(5)};
\draw[col0,line width=1.2pt] (-3.000,-0.346) -- (-1.000,-0.346);
\draw[col1,line width=1.2pt] (-1.000,-0.346) -- (0.000,1.386);
\draw[col2,line width=1.2pt] (-1.000,-0.346) -- (1.000,-0.346);
\draw[col3,line width=1.2pt] (0.000,1.386) -- (1.000,-0.346);
\draw[col4,line width=1.2pt] (1.000,-0.346) -- (3.000,-0.346);
\fill[black!17,opacity=0.45] (-3.000,-0.346) circle (1.000);
\fill[black!17,opacity=0.45] (-1.000,-0.346) circle (1.000);
\fill[black!17,opacity=0.45] (0.000,1.386) circle (1.000);
\fill[black!17,opacity=0.45] (1.000,-0.346) circle (1.000);
\fill[black!17,opacity=0.45] (3.000,-0.346) circle (1.000);
\draw[gray,line width=0.8pt] (-3.000,-0.346) circle (1.000);
\draw[gray,line width=0.8pt] (-1.000,-0.346) circle (1.000);
\draw[gray,line width=0.8pt] (0.000,1.386) circle (1.000);
\draw[gray,line width=0.8pt] (1.000,-0.346) circle (1.000);
\draw[gray,line width=0.8pt] (3.000,-0.346) circle (1.000);
\draw[diskhull]
  (-3.000,-1.346)
  -- (3.000,-1.346)
  arc[start angle=-90.0,end angle=60.0,radius=1.000]
  -- (0.500,2.252)
  arc[start angle=60.0,end angle=120.0,radius=1.000]
  -- (-3.500,0.520)
  arc[start angle=120.0,end angle=270.0,radius=1.000];
\filldraw[black] (-3.000,-0.346) circle (0.05);
\filldraw[black] (-1.000,-0.346) circle (0.05);
\filldraw[black] (0.000,1.386) circle (0.05);
\filldraw[black] (1.000,-0.346) circle (0.05);
\filldraw[black] (3.000,-0.346) circle (0.05);
\end{tikzpicture}
&
\begin{tikzpicture}[scale=\graphscale, line cap=round, line join=round]
\node[anchor=north west, inner sep=1pt, text=black!65, font=\scriptsize] at (-4.6,2.9) {(6)};
\draw[col0,line width=1.2pt] (-3.493,0.000) -- (-1.493,0.000);
\draw[col1,line width=1.2pt] (-1.493,0.000) -- (0.507,0.000);
\draw[col2,line width=1.2pt] (0.507,0.000) -- (2.239,1.000);
\draw[col3,line width=1.2pt] (2.239,1.000) -- (2.239,-1.000);
\draw[col4,line width=1.2pt] (0.507,0.000) -- (2.239,-1.000);
\fill[black!17,opacity=0.45] (-3.493,0.000) circle (1.000);
\fill[black!17,opacity=0.45] (-1.493,0.000) circle (1.000);
\fill[black!17,opacity=0.45] (0.507,0.000) circle (1.000);
\fill[black!17,opacity=0.45] (2.239,1.000) circle (1.000);
\fill[black!17,opacity=0.45] (2.239,-1.000) circle (1.000);
\draw[gray,line width=0.8pt] (-3.493,0.000) circle (1.000);
\draw[gray,line width=0.8pt] (-1.493,0.000) circle (1.000);
\draw[gray,line width=0.8pt] (0.507,0.000) circle (1.000);
\draw[gray,line width=0.8pt] (2.239,1.000) circle (1.000);
\draw[gray,line width=0.8pt] (2.239,-1.000) circle (1.000);
\draw[diskhull]
  (-3.665,-0.985)
  -- (2.067,-1.985)
  arc[start angle=-99.9,end angle=0.0,radius=1.000]
  -- (3.239,1.000)
  arc[start angle=0.0,end angle=99.9,radius=1.000]
  -- (-3.665,0.985)
  arc[start angle=99.9,end angle=260.1,radius=1.000];
\filldraw[black] (-3.493,0.000) circle (0.05);
\filldraw[black] (-1.493,0.000) circle (0.05);
\filldraw[black] (0.507,0.000) circle (0.05);
\filldraw[black] (2.239,1.000) circle (0.05);
\filldraw[black] (2.239,-1.000) circle (0.05);
\end{tikzpicture}
&
\begin{tikzpicture}[scale=\graphscale, line cap=round, line join=round]
\node[anchor=north west, inner sep=1pt, text=black!65, font=\scriptsize] at (-2.9,3.0) {(7)};
\draw[col0,line width=1.2pt] (-1.732,1.000) -- (-1.732,-1.000);
\draw[col1,line width=1.2pt] (-1.732,-1.000) -- (0.000,0.000);
\draw[col2,line width=1.2pt] (0.000,0.000) -- (-1.732,1.000);
\draw[col3,line width=1.2pt] (0.000,0.000) -- (1.732,1.000);
\draw[col4,line width=1.2pt] (1.732,1.000) -- (1.732,-1.000);
\draw[col5,line width=1.2pt] (1.732,-1.000) -- (0.000,0.000);
\fill[black!17,opacity=0.45] (-1.732,1.000) circle (1.000);
\fill[black!17,opacity=0.45] (-1.732,-1.000) circle (1.000);
\fill[black!17,opacity=0.45] (0.000,0.000) circle (1.000);
\fill[black!17,opacity=0.45] (1.732,1.000) circle (1.000);
\fill[black!17,opacity=0.45] (1.732,-1.000) circle (1.000);
\draw[gray,line width=0.8pt] (-1.732,1.000) circle (1.000);
\draw[gray,line width=0.8pt] (-1.732,-1.000) circle (1.000);
\draw[gray,line width=0.8pt] (0.000,0.000) circle (1.000);
\draw[gray,line width=0.8pt] (1.732,1.000) circle (1.000);
\draw[gray,line width=0.8pt] (1.732,-1.000) circle (1.000);
\draw[diskhull]
  (-1.732,-2.000)
  -- (1.732,-2.000)
  arc[start angle=-90.0,end angle=0.0,radius=1.000]
  -- (2.732,1.000)
  arc[start angle=0.0,end angle=90.0,radius=1.000]
  -- (-1.732,2.000)
  arc[start angle=90.0,end angle=180.0,radius=1.000]
  -- (-2.732,-1.000)
  arc[start angle=-180.0,end angle=-90.0,radius=1.000];
\filldraw[black] (-1.732,1.000) circle (0.05);
\filldraw[black] (-1.732,-1.000) circle (0.05);
\filldraw[black] (0.000,0.000) circle (0.05);
\filldraw[black] (1.732,1.000) circle (0.05);
\filldraw[black] (1.732,-1.000) circle (0.05);
\end{tikzpicture}
&
\begin{tikzpicture}[scale=\graphscale, line cap=round, line join=round]
\node[anchor=north west, inner sep=1pt, text=black!65, font=\scriptsize] at (-3.5,2.3) {(8)};
\draw[col0,line width=1.2pt] (-2.400,0.000) -- (-0.400,0.000);
\draw[col1,line width=1.2pt] (-0.400,0.000) -- (1.600,0.000);
\draw[col2,line width=1.2pt] (-0.400,0.000) -- (0.600,1.732);
\draw[col3,line width=1.2pt] (-0.400,0.000) -- (0.600,-1.732);
\draw[col4,line width=1.2pt] (1.600,0.000) -- (0.600,1.732);
\draw[col5,line width=1.2pt] (1.600,0.000) -- (0.600,-1.732);
\fill[black!17,opacity=0.45] (-2.400,0.000) circle (1.000);
\fill[black!17,opacity=0.45] (-0.400,0.000) circle (1.000);
\fill[black!17,opacity=0.45] (1.600,0.000) circle (1.000);
\fill[black!17,opacity=0.45] (0.600,1.732) circle (1.000);
\fill[black!17,opacity=0.45] (0.600,-1.732) circle (1.000);
\draw[gray,line width=0.8pt] (-2.400,0.000) circle (1.000);
\draw[gray,line width=0.8pt] (-0.400,0.000) circle (1.000);
\draw[gray,line width=0.8pt] (1.600,0.000) circle (1.000);
\draw[gray,line width=0.8pt] (0.600,1.732) circle (1.000);
\draw[gray,line width=0.8pt] (0.600,-1.732) circle (1.000);
\draw[diskhull]
  (-2.900,-0.866)
  -- (0.100,-2.598)
  arc[start angle=-120.0,end angle=-30.0,radius=1.000]
  -- (2.466,-0.500)
  arc[start angle=-30.0,end angle=30.0,radius=1.000]
  -- (1.466,2.232)
  arc[start angle=30.0,end angle=120.0,radius=1.000]
  -- (-2.900,0.866)
  arc[start angle=120.0,end angle=240.0,radius=1.000];
\filldraw[black] (-2.400,0.000) circle (0.05);
\filldraw[black] (-0.400,0.000) circle (0.05);
\filldraw[black] (1.600,0.000) circle (0.05);
\filldraw[black] (0.600,1.732) circle (0.05);
\filldraw[black] (0.600,-1.732) circle (0.05);
\end{tikzpicture}
\\[0.3cm]

\begin{tikzpicture}[scale=\graphscale, line cap=round, line join=round]
\node[anchor=north west, inner sep=1pt, text=black!65, font=\scriptsize] at (-4.0,3.3) {(9)};
\draw[col0,line width=1.2pt] (-2.986,0.000) -- (-0.986,0.000);
\draw[col1,line width=1.2pt] (-0.986,0.000) -- (0.746,1.000);
\draw[col2,line width=1.2pt] (-0.986,0.000) -- (0.746,-1.000);
\draw[col3,line width=1.2pt] (0.746,1.000) -- (0.746,-1.000);
\draw[col4,line width=1.2pt] (0.746,1.000) -- (2.478,0.000);
\draw[col5,line width=1.2pt] (0.746,-1.000) -- (2.478,0.000);
\fill[black!17,opacity=0.45] (-2.986,0.000) circle (1.000);
\fill[black!17,opacity=0.45] (-0.986,0.000) circle (1.000);
\fill[black!17,opacity=0.45] (0.746,1.000) circle (1.000);
\fill[black!17,opacity=0.45] (0.746,-1.000) circle (1.000);
\fill[black!17,opacity=0.45] (2.478,0.000) circle (1.000);
\draw[gray,line width=0.8pt] (-2.986,0.000) circle (1.000);
\draw[gray,line width=0.8pt] (-0.986,0.000) circle (1.000);
\draw[gray,line width=0.8pt] (0.746,1.000) circle (1.000);
\draw[gray,line width=0.8pt] (0.746,-1.000) circle (1.000);
\draw[gray,line width=0.8pt] (2.478,0.000) circle (1.000);
\draw[diskhull]
  (-3.245,-0.966)
  -- (0.487,-1.966)
  arc[start angle=-105.0,end angle=-60.0,radius=1.000]
  -- (2.978,-0.866)
  arc[start angle=-60.0,end angle=60.0,radius=1.000]
  -- (1.246,1.866)
  arc[start angle=60.0,end angle=105.0,radius=1.000]
  -- (-3.245,0.966)
  arc[start angle=105.0,end angle=255.0,radius=1.000];
\filldraw[black] (-2.986,0.000) circle (0.05);
\filldraw[black] (-0.986,0.000) circle (0.05);
\filldraw[black] (0.746,1.000) circle (0.05);
\filldraw[black] (0.746,-1.000) circle (0.05);
\filldraw[black] (2.478,0.000) circle (0.05);
\end{tikzpicture}
&
\begin{tikzpicture}[scale=\graphscale, line cap=round, line join=round]
\node[anchor=north west, inner sep=1pt, text=black!65, font=\scriptsize] at (-2.8,3.3) {$(10)$};
\draw[col0,line width=1.2pt] (-1.546,-1.000) -- (0.454,-1.000);
\draw[col1,line width=1.2pt] (-1.546,-1.000) -- (-1.546,1.000);
\draw[col2,line width=1.2pt] (0.454,-1.000) -- (0.454,1.000);
\draw[col3,line width=1.2pt] (0.454,-1.000) -- (2.186,0.000);
\draw[col4,line width=1.2pt] (-1.546,1.000) -- (0.454,1.000);
\draw[col5,line width=1.2pt] (0.454,1.000) -- (2.186,0.000);
\fill[black!17,opacity=0.45] (-1.546,-1.000) circle (1.000);
\fill[black!17,opacity=0.45] (0.454,-1.000) circle (1.000);
\fill[black!17,opacity=0.45] (-1.546,1.000) circle (1.000);
\fill[black!17,opacity=0.45] (0.454,1.000) circle (1.000);
\fill[black!17,opacity=0.45] (2.186,0.000) circle (1.000);
\draw[gray,line width=0.8pt] (-1.546,-1.000) circle (1.000);
\draw[gray,line width=0.8pt] (0.454,-1.000) circle (1.000);
\draw[gray,line width=0.8pt] (-1.546,1.000) circle (1.000);
\draw[gray,line width=0.8pt] (0.454,1.000) circle (1.000);
\draw[gray,line width=0.8pt] (2.186,0.000) circle (1.000);
\draw[diskhull]
  (-1.546,-2.000)
  -- (0.454,-2.000)
  arc[start angle=-90.0,end angle=-60.0,radius=1.000]
  -- (2.686,-0.866)
  arc[start angle=-60.0,end angle=60.0,radius=1.000]
  -- (0.954,1.866)
  arc[start angle=60.0,end angle=90.0,radius=1.000]
  -- (-1.546,2.000)
  arc[start angle=90.0,end angle=180.0,radius=1.000]
  -- (-2.546,-1.000)
  arc[start angle=-180.0,end angle=-90.0,radius=1.000];
\filldraw[black] (-1.546,-1.000) circle (0.05);
\filldraw[black] (0.454,-1.000) circle (0.05);
\filldraw[black] (-1.546,1.000) circle (0.05);
\filldraw[black] (0.454,1.000) circle (0.05);
\filldraw[black] (2.186,0.000) circle (0.05);
\end{tikzpicture}
&
\begin{tikzpicture}[scale=\graphscale, line cap=round, line join=round]
\node[anchor=north west, inner sep=1pt, text=black!65, font=\scriptsize] at (-2.9,2.9) {$(11)$};
\draw[col0,line width=1.2pt] (0.000,1.701) -- (-1.618,0.526);
\draw[col1,line width=1.2pt] (0.000,1.701) -- (1.618,0.526);
\draw[col2,line width=1.2pt] (-1.618,0.526) -- (-1.000,-1.376);
\draw[col3,line width=1.2pt] (-1.000,-1.376) -- (1.000,-1.376);
\draw[col4,line width=1.2pt] (1.000,-1.376) -- (1.618,0.526);
\fill[black!17,opacity=0.45] (0.000,1.701) circle (1.000);
\fill[black!17,opacity=0.45] (-1.618,0.526) circle (1.000);
\fill[black!17,opacity=0.45] (-1.000,-1.376) circle (1.000);
\fill[black!17,opacity=0.45] (1.000,-1.376) circle (1.000);
\fill[black!17,opacity=0.45] (1.618,0.526) circle (1.000);
\draw[gray,line width=0.8pt] (0.000,1.701) circle (1.000);
\draw[gray,line width=0.8pt] (-1.618,0.526) circle (1.000);
\draw[gray,line width=0.8pt] (-1.000,-1.376) circle (1.000);
\draw[gray,line width=0.8pt] (1.000,-1.376) circle (1.000);
\draw[gray,line width=0.8pt] (1.618,0.526) circle (1.000);
\draw[diskhull]
  (-2.569,0.217)
  -- (-1.951,-1.685)
  arc[start angle=-162.0,end angle=-90.0,radius=1.000]
  -- (1.000,-2.376)
  arc[start angle=-90.0,end angle=-18.0,radius=1.000]
  -- (2.569,0.217)
  arc[start angle=-18.0,end angle=54.0,radius=1.000]
  -- (0.588,2.510)
  arc[start angle=54.0,end angle=126.0,radius=1.000]
  -- (-2.206,1.335)
  arc[start angle=126.0,end angle=198.0,radius=1.000];
\filldraw[black] (0.000,1.701) circle (0.05);
\filldraw[black] (-1.618,0.526) circle (0.05);
\filldraw[black] (-1.000,-1.376) circle (0.05);
\filldraw[black] (1.000,-1.376) circle (0.05);
\filldraw[black] (1.618,0.526) circle (0.05);
\end{tikzpicture}
&
\begin{tikzpicture}[scale=\graphscale, line cap=round, line join=round]
\node[anchor=north west, inner sep=1pt, text=black!65, font=\scriptsize] at (-2.9,3.5) {$(12)$};
\draw[col0,line width=1.2pt] (-2.000,-0.693) -- (0.000,-0.693);
\draw[col1,line width=1.2pt] (-2.000,-0.693) -- (-1.000,1.039);
\draw[col2,line width=1.2pt] (0.000,-0.693) -- (2.000,-0.693);
\draw[col3,line width=1.2pt] (0.000,-0.693) -- (-1.000,1.039);
\draw[col4,line width=1.2pt] (0.000,-0.693) -- (1.000,1.039);
\draw[col5,line width=1.2pt] (2.000,-0.693) -- (1.000,1.039);
\draw[col6,line width=1.2pt] (-1.000,1.039) -- (1.000,1.039);
\fill[black!17,opacity=0.45] (-2.000,-0.693) circle (1.000);
\fill[black!17,opacity=0.45] (0.000,-0.693) circle (1.000);
\fill[black!17,opacity=0.45] (2.000,-0.693) circle (1.000);
\fill[black!17,opacity=0.45] (-1.000,1.039) circle (1.000);
\fill[black!17,opacity=0.45] (1.000,1.039) circle (1.000);
\draw[gray,line width=0.8pt] (-2.000,-0.693) circle (1.000);
\draw[gray,line width=0.8pt] (0.000,-0.693) circle (1.000);
\draw[gray,line width=0.8pt] (2.000,-0.693) circle (1.000);
\draw[gray,line width=0.8pt] (-1.000,1.039) circle (1.000);
\draw[gray,line width=0.8pt] (1.000,1.039) circle (1.000);
\draw[diskhull]
  (-2.000,-1.693)
  -- (2.000,-1.693)
  arc[start angle=-90.0,end angle=30.0,radius=1.000]
  -- (1.866,1.539)
  arc[start angle=30.0,end angle=90.0,radius=1.000]
  -- (-1.000,2.039)
  arc[start angle=90.0,end angle=150.0,radius=1.000]
  -- (-2.866,-0.193)
  arc[start angle=150.0,end angle=270.0,radius=1.000];
\filldraw[black] (-2.000,-0.693) circle (0.05);
\filldraw[black] (0.000,-0.693) circle (0.05);
\filldraw[black] (2.000,-0.693) circle (0.05);
\filldraw[black] (-1.000,1.039) circle (0.05);
\filldraw[black] (1.000,1.039) circle (0.05);
\end{tikzpicture}
\\[0.7cm]

\multicolumn{4}{c}{
\begin{tikzpicture}[scale=\graphscale, line cap=round, line join=round]
\node[anchor=north west, inner sep=1pt, text=black!65, font=\scriptsize] at (-4.6,3.2) {(13)};

\coordinate (C1) at (-4.000,0.000);
\coordinate (C2) at (-2.268,1.000);
\coordinate (C3) at (-0.536,0.000);
\coordinate (C4) at (1.196,1.000);
\coordinate (C5) at (2.928,0.000);

\draw[col0,line width=1.2pt] (C1) -- (C2);
\draw[col1,line width=1.2pt] (C2) -- (C3);
\draw[col2,line width=1.2pt] (C3) -- (C4);
\draw[col3,line width=1.2pt] (C4) -- (C5);

\fill[black!17,opacity=0.45] (C1) circle (1.000);
\fill[black!17,opacity=0.45] (C2) circle (1.000);
\fill[black!17,opacity=0.45] (C3) circle (1.000);
\fill[black!17,opacity=0.45] (C4) circle (1.000);
\fill[black!17,opacity=0.45] (C5) circle (1.000);

\draw[gray,line width=0.8pt] (C1) circle (1.000);
\draw[gray,line width=0.8pt] (C2) circle (1.000);
\draw[gray,line width=0.8pt] (C3) circle (1.000);
\draw[gray,line width=0.8pt] (C4) circle (1.000);
\draw[gray,line width=0.8pt] (C5) circle (1.000);

\draw[diskhull]
  (-4.000,-1.000)
  -- (2.928,-1.000)
  arc[start angle=-90,end angle=60,radius=1.000]
  -- (1.696,1.866)
  arc[start angle=60,end angle=90,radius=1.000]
  -- (-2.268,2.000)
  arc[start angle=90,end angle=120,radius=1.000]
  -- (-4.500,0.866)
  arc[start angle=120,end angle=270,radius=1.000];

\filldraw[black] (C1) circle (0.05);
\filldraw[black] (C2) circle (0.05);
\filldraw[black] (C3) circle (0.05);
\filldraw[black] (C4) circle (0.05);
\filldraw[black] (C5) circle (0.05);
\end{tikzpicture}
}
\\[0.1cm]

\end{tabular}
\caption{Realisations
\({\bf c}_1,\ldots,{\bf c}_{13}\) belonging to the fixed hull and contact
classes \(X_1,\ldots,X_{13}\) displayed in the finite verification.}
\label{fig:n5-graphs}
\end{figure}
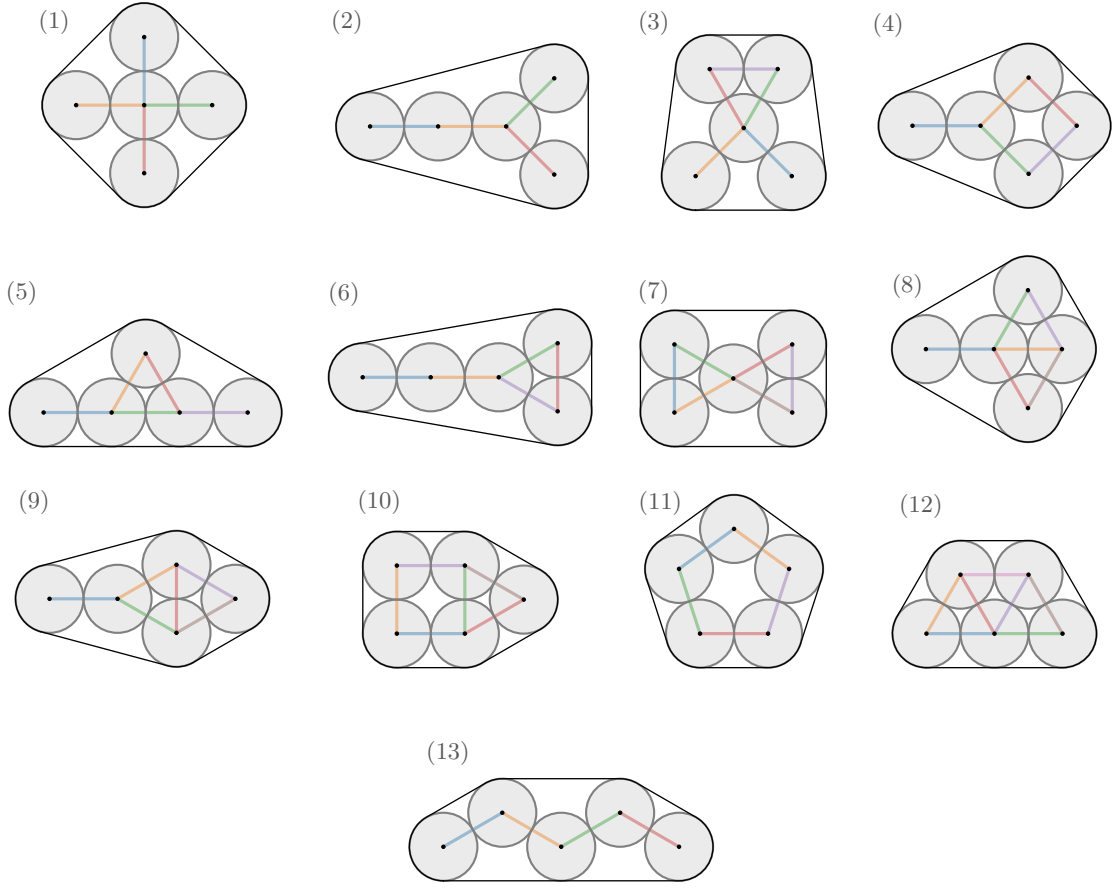

\section{The five disk perimeter problem}
\label{sec:n5}

As an application, we use the preceding theory to solve the five disk
perimeter problem. The computational details are described in
Appendix~\ref{app:computational-verification}, and the accompanying code,
input data, and verification logs are available in \cite{repo}.

\subsection{Graph enumeration and realisability filtering}
\label{sec:enum-filter}
It suffices to consider connected contact graphs, since any disconnected
minimising configuration can be replaced by an incident configuration with an
additional contact and no larger centre hull perimeter.

A contact graph of an equal radius hard disk configuration is called a penny
graph. Such graphs are planar, have maximum degree at most \(6\), and satisfy
Harborth's edge bound \cite{harborth}
$
|E|\le\bigl\lfloor 3n-\sqrt{12n-3}\,\bigr\rfloor,
$
which gives \(|E|\le 7\) for \(n=5\). To generate the candidate list, we use \texttt{plantri}, reduce the output
up to isomorphism with \texttt{nauty}, restrict the maximum degree and edge
count with \texttt{pickg} \cite{brinkmannmckay2007,mckaypiperno2014}, and
retain only those graphs realised by centre configurations in which every
active edge has length two and every non edge distance is strictly greater
than two. This yields exactly thirteen contact classes, denoted by \(X_1,\ldots,X_{13}\), in agreement with the value \(a(5)=13\) in OEISA085632, which counts connected penny graphs on five vertices up to graph isomorphism \cite{oeisA085632}. The enumeration and the associated fixed hull and contact classes within each \(X_i\) are documented in \cite{repo}.

\subsection{Contact and hull classes}
The perimeter analysis is organised at two levels. First, configurations are
classified by their contact class, determined by the contact graph. Each
contact class then splits into finitely many fixed hull and contact classes,
according to the hull vertex set and cyclic hull order, as discussed in Example~\ref{ex:same-contact-different-hull}.

Figure~\ref{fig:n5-graphs} displays one representative of each contact class \(X_1,\ldots,X_{13}\), chosen in general to exhibit a parallel leaf realisation of the contact graph. These representatives are selected for geometric clarity and are not a priori canonical. Contact classes admit a unique hull refinement, whereas others admit
multiple hull classes, as illustrated by \(X_{12}\) and \(X_{13}\),
respectively.

\begin{theorem}
\label{thm:n5-final}
The minimum perimeter among all configurations of five unit hard disks is
\[
10+2\pi.
\]
It is attained precisely by configurations in the contact classes
\(X_{10},X_{11},X_{12}\). The minimising configurations in \(X_{10}\) and
\(X_{11}\) admit local perimeter preserving admissible flexes, while the
minimising configuration in \(X_{12}\) is rigid modulo rigid motions.
\end{theorem}

\begin{proof}
By Theorem~\ref{thm:existence-global-min}, a global minimiser exists. 

{\bf{Finite reduction.}} By Subsection~\ref{sec:enum-filter}, every contact class realised by a five
disk minimiser is one of \(X_1,\ldots,X_{13}\). The complete enumeration and
filtering are recorded in \cite{repo}.

We now exclude the classes that cannot contain a perimeter minimiser.

{\bf{The circular types \(X_1\) and \(X_7\).}} For \(X_1\), place the centre of the unique non hull disk at the origin. The centres of the four hull disks then lie on the circle of radius two centred at the origin. For a fixed cyclic hull order, let
\(\theta_1,\ldots,\theta_4\) be the consecutive pivot gap angles. The relative
interior of the class is parametrised by the convex domain
\[
\Theta
=
\left\{
\theta\in\mathbb R^4:
\theta_i>\frac{\pi}{3},\quad
\sum_{i=1}^4\theta_i=2\pi
\right\}.
\]
Since
\[
\theta_i
=
2\pi-\sum_{j\ne i}\theta_j
<
2\pi-3\cdot\frac{\pi}{3}
=
\pi,
\]
it follows that
\[
\frac{\theta_i}{2}
\in
\left(
\frac{\pi}{6},
\frac{\pi}{2}
\right)
\qquad
(i=1,\ldots,4).
\]
Using the chord length formula, the centre hull perimeter is
\[
p(\theta)
=
\sum_{i=1}^4
4\sin\left(\frac{\theta_i}{2}\right).
\]

The Hessian of \(p\) in \(\mathbb R^4\) is
\[
D^2p(\theta)
=
\operatorname{diag}\!\left(
-\sin\left(\frac{\theta_1}{2}\right),
\ldots,
-\sin\left(\frac{\theta_4}{2}\right)
\right).
\]
Restricting to
\[
T_\theta\Theta
=
\left\{
\eta\in\mathbb R^4:
\sum_{i=1}^4\eta_i=0
\right\},
\]

Since each diagonal entry is negative, \(D^2p(\theta)\) is negative definite
and therefore so is its restriction to
\[
T_\theta\Theta
=
\left\{
\eta\in\mathbb R^4:
\sum_{i=1}^4\eta_i=0
\right\}.
\]

Hence \(p\) is strictly concave on \(\Theta\) and has no relative interior
local minimum. 

For \(X_7\), two additional active contacts impose two angular equalities, so
the relative interior of each hull class is the relative interior of an affine
face of \(\Theta\). The restriction of a strictly concave function to an affine
face remains strictly concave. Hence \(p\) has no relative interior local
minimum on the relative interior of any hull class assigned to \(X_7\).

Consequently, no hull class assigned to \(X_1\) or \(X_7\) contains a local
minimiser.

{\bf{The hull leaf types.}}
Every hull class of
\[
X_2,X_3,X_5,X_6,X_8,X_9,X_{13}
\]
has a hull leaf. By Theorem~\ref{thm:leaf-interior-descent}, every realisation
for which the leaf parallelism condition fails admits an admissible first
order perimeter descent and is therefore not a local minimiser.

For \(X_{13}\), configurations that do not define convex polygonal hulls,
including the completely horizontal one, are
excluded directly by Theorem~\ref{thm:leaf-interior-descent}.

For the hull classes of
\[
X_2,X_3,X_5,X_6,X_8,X_9,
\]
the verification in \cite{repo} identifies all parallel leaf realisations,
modulo rigid motions. Each such class is critical and has a negative intrinsic Hessian eigenvalue on the reduced rolling space. Proposition~5.7 excludes these realisations.
Consequently, no hull class assigned to these contact classes contains a local minimiser.

Table~\ref{tab:n5-classes} records the rank, reduced rolling dimension, and
certified intrinsic eigenvalue data for the displayed representatives in Figure~\ref{fig:n5-graphs}. 

{\bf{The exceptional case \(X_4\).}} The nonparallel leaf locus of every class assigned to \(X_4\) is excluded by
Theorem~\ref{thm:leaf-interior-descent}. Let \({\bf c}\) be a relative
interior realisation in the parallel leaf locus. Up to rigid motion,
reflection, and relabelling, it has the form
\[
c_1=(-a,0),\qquad
c_2=(0,b),\qquad
c_3=(a,0),\qquad
c_4=(0,-b),\qquad
c_0=(-a-2,0),
\]
where
\[
a^2+b^2=4.
\]

The strict inactive conditions
\[
\lVert c_1-c_3\rVert>2,
\qquad
\lVert c_2-c_4\rVert>2
\]
are equivalent to
\[
1<a<\sqrt3.
\]
Set
\[
a(t)=a-t,
\qquad
b(t)=\sqrt{4-a(t)^2},
\qquad
c_0(t)=(-a(t)-2,0),
\]
and define \(c_1(t),\ldots,c_4(t)\) by the same coordinate formulas. The five
active contacts are preserved identically. For all sufficiently small
\(t>0\), the inactive inequalities remain strict, since
\[
a(t)>1,
\qquad
b(t)>b>1,
\]
\[
\lVert c_0(t)-c_3(t)\rVert=2a(t)+2>2,
\]
and
\[
\lVert c_0(t)-c_j(t)\rVert^2=8+4a(t)>4,
\qquad j=2,4.
\]
Hence \({\bf c}(t)\in\mathcal D_5\).

The quadrilateral
\[
c_0(t)c_2(t)c_3(t)c_4(t)
\]
is convex, and its intersection with the \(x\)-axis contains
\(c_1(t)=(-a(t),0)\) in the open segment joining \(c_0(t)\) to \(c_3(t)\).
Thus the hull vertices remain
\[
c_0(t),c_2(t),c_3(t),c_4(t).
\]
Moreover,
\[
\lVert c_2(t)-c_3(t)\rVert
=
\lVert c_3(t)-c_4(t)\rVert
=
2
\]
and
\[
\lVert c_0(t)-c_2(t)\rVert
=
\lVert c_0(t)-c_4(t)\rVert
=
2\sqrt{2+a(t)}.
\]
Therefore
\[
p(t)
=
\Per(P({\bf c}(t)))
=
4+4\sqrt{2+a(t)}
\]
and
\[
p'(0)
=
-\frac{2}{\sqrt{2+a}}
<0.
\]
Every relative interior point of the parallel leaf locus therefore admits strict first order descent.

As the parameter approaches to \(a=1\), the additional contact
$
\|c_1-c_3\|=2
$
forms, and the limiting configuration belongs to a hull class of \(X_8\). When \(a=\sqrt3\), the additional contact
$
\|c_2-c_4\|=2
$
forms, and the resulting configuration belongs to a hull class of \(X_9\). These boundary classes are hull leaf types and therefore excluded. Hence no fixed hull and
contact class assigned to \(X_4\) contains a local minimiser.

{\bf{The surviving classes.}}
By the preceding exclusions, every global minimiser belongs to one of the
contact classes
\[
X_{10},X_{11},X_{12}.
\]

Each has a unique hull class, considered modulo rigid motions.
For \(X_{10}\) and \(X_{11}\), every hull edge is an active contact of length
\(2\), so
$
\Per(P({\bf c}))=10.
$

For \(X_{12}\), the certified centre hull is an isosceles trapezoid with side
lengths \(4,2,2,2\), and therefore
$
\Per(P({\bf c}))=4+2+2+2=10.
$

For the realisations
\({\bf c}_{10},{\bf c}_{11},{\bf c}_{12}\) in
Figure~\ref{fig:n5-graphs},
\[
\operatorname{rank}A({\bf c}_{10})=6,\qquad
\operatorname{rank}A({\bf c}_{11})=5,\qquad
\operatorname{rank}A({\bf c}_{12})=7.
\]

 Hence the corresponding reduced rolling spaces have dimensions \(1\), \(2\), and \(0\), respectively. Proposition~\ref{prop:flat-class} shows that the minimising configurations in \(X_{10}\) and \(X_{11}\) form local perimeter preserving admissible families of dimensions \(1\) and \(2\), respectively, while Proposition~\ref{prop:rigid-reduced-class} shows that the minimising configuration in \(X_{12}\) is rigid modulo rigid motions.

\end{proof}

\begin{table}[ht]
\caption{Ranks of the contact operators \(A({\bf c}_i)\), dimensions of
the reduced rolling spaces, exclusion outcomes, and smallest intrinsic
eigenvalues, where defined, for the thirteen five disk classes \(X_i\).
Representative realisations \({\bf c}_i\) are shown in
Figure~\ref{fig:n5-graphs}. Detailed spectral and perimeter data for the
surviving classes \(X_{10},X_{11},X_{12}\) are given in
Table~\ref{tab:five-disk-spectrum}. The collinear class \(X_{13}\) has no
polygonal intrinsic Hessian and is excluded by the hull leaf obstruction,
Theorem~\ref{thm:leaf-interior-descent}.}
\centering
\renewcommand{\arraystretch}{1.15}
\setlength{\tabcolsep}{8pt}
\begin{tabular}{ccccc}
\toprule
Class & \(\rank A({\bf c})\) & \(\dim\mathcal R_{\bf c}\) & Outcome
& \(\lambda_{\min}(\mathcal H_{\mathrm{intr}})\) \\
\midrule
\(X_1\) & \(4\) & \(3\) & excluded & \([-3.5859,-3.5858]\) \\
\(X_2\) & \(4\) & \(3\) & excluded & \([-6.6404,-6.6403]\) \\
\(X_3\) & \(5\) & \(2\) & excluded & \([-5.7482,-5.7481]\) \\
\(X_4\) & \(5\) & \(2\) & excluded & \([-3.0960,-3.0959]\) \\
\(X_5\) & \(5\) & \(2\) & excluded & \([-12.204,-7.2100]\) \\
\(X_6\) & \(5\) & \(2\) & excluded & \([-6.3477,-6.3476]\) \\
\(X_7\) & \(6\) & \(1\) & excluded & \([-6.9283,-6.9282]\) \\
\(X_8\) & \(6\) & \(1\) & excluded & \([-2.4425,-2.4424]\) \\
\(X_9\) & \(6\) & \(1\) & excluded & \([-4.3207,-4.3206]\) \\
\(X_{10}\) & \(6\) & \(1\) & \textbf{survivor} & \(0\) \\
\(X_{11}\) & \(5\) & \(2\) & \textbf{survivor} & \(0\) \\
\(X_{12}\) & \(7\) & \(0\) & \textbf{survivor} & \(\mathrm{n/a}\) \\
\(X_{13}\) & \(4\) & \(3\) & excluded & \(\mathrm{n/a}\) \\
\bottomrule
\end{tabular}

\label{tab:n5-classes}
\end{table}

\begin{table}[ht]
\caption{Structural data for the three minimising five disk classes, where
\(\mathcal R_{\bf c}=\Roll({\bf c})\cap R({\bf c})^\perp\) and
\(\dim R({\bf c})=3\). For \(X_{10}\) and \(X_{11}\), every hull edge is an
active contact edge, so the perimeter is constant on the fixed hull and contact
class and the intrinsic Hessian vanishes on \(\mathcal R_{\bf c}\). The class \(X_{12}\) is rigid since
\(\Roll({\bf c})=R({\bf c})\), so
\(\mathcal R_{\bf c}=\{0\}\) and
\(\operatorname{spec}(\mathcal H_{\mathrm{intr}})=\varnothing\).}
\centering
\begin{tabular}{cccccccc}
\toprule
Class & \(\rank A({\bf c})\) & \(\dim\Roll({\bf c})\) & \(\dim\mathcal R_{\bf c}\) & \(\operatorname{spec}(\mathcal H_{\mathrm{intr}})\) & \(\Per(P({\bf c}))\) & \(\Per({\bf c})\) & Outcome \\
\midrule
\(X_{10}\) & \(6\) & \(4\) & \(1\) & \(\{0\}\) & \(10\) & \(10+2\pi\) & flat survivor \\
\(X_{11}\) & \(5\) & \(5\) & \(2\) & \(\{0,0\}\) & \(10\) & \(10+2\pi\) & flat survivor \\
\(X_{12}\) & \(7\) & \(3\) & \(0\) & \(\varnothing\) & \(10\) & \(10+2\pi\) & rigid survivor \\
\bottomrule
\end{tabular}

\label{tab:five-disk-spectrum}
\end{table}

\appendix
\section{Four disk spectral prototypes}
\label{sec:examples-four-disks}
We present four disks configuration examples illustrating the intrinsic spectrum
introduced in Section~\ref{sec:second-variation}. The point is not merely to
detect first order flexes, but to test the remaining admissible directions
modulo rigid motions by the second variation of the perimeter. Negative eigenvalues give second order descent directions, zero eigenvalues give
flat directions, and the absence of reduced directions gives rigidity modulo
rigid motions. The examples below show these alternatives in explicit
coordinates: failure of first order class criticality, flat second order
degeneracy, and rigidity modulo rigid motions. Figure~\ref{fig:34} illustrates the corresponding configurations.

\subsection{The square flat family through second order degeneracy}
\label{ex:square-degenerate}
Consider the square configuration
\[
{\bf c}_1=(0,0),\quad
{\bf c}_2=(2,0),\quad
{\bf c}_3=(2,2),\quad
{\bf c}_4=(0,2),
\]
with contact graph
\[
E(G)=\bigl\{\{1,2\},\{2,3\},\{3,4\},\{4,1\}\bigr\}
\]
and hull cycle
\[
\mathcal B=\{(1,2),(2,3),(3,4),(4,1)\}.
\]
This example gives a class critical configuration whose zero intrinsic
eigenvalue is realised by an actual perimeter preserving family. A direct
computation gives
\[
A({\bf c})^\top\lambda=\nabla\widetilde\Per({\bf c}),
\qquad
\lambda=(1,1,1,1),
\]
so \({\bf c}\) is class critical.
The fixed hull and contact class locally contains the rhombus family
\[
{\bf c}_1(\theta)=(0,0),\qquad
{\bf c}_2(\theta)=(2,0),\qquad
{\bf c}_4(\theta)=(2\cos\theta,2\sin\theta),
\]
\[
{\bf c}_3(\theta)={\bf c}_2(\theta)+{\bf c}_4(\theta),
\qquad
\theta\in\left(\frac{\pi}{3},\frac{2\pi}{3}\right).
\]
For \(\theta\) near \(\pi/2\), this remains in the same fixed hull and contact
class. The two noncontact diagonals have lengths \(4\sin(\theta/2)\) and
\(4\cos(\theta/2)\), so they remain strictly larger than \(2\) on this
interval. Each hull edge has length \(2\), hence
\[
\Per({\bf c}(\theta))=8+2\pi.
\]
Thus the perimeter is constant along this one parameter family.
A direct computation gives
\[
\operatorname{rank} A({\bf c})=4,
\qquad
\dim\Roll({\bf c})=8-4=4.
\]
Since \(\dim R({\bf c})=3\) and \(R({\bf c})\subset\Roll({\bf c})\),
\[
\dim\bigl(\Roll({\bf c})\cap R({\bf c})^\perp\bigr)=1.
\]
On this one dimensional reduced admissible space,
\[
\mathcal H_{\mathrm{intr}}=[0],
\qquad
\operatorname{spec}(\mathcal H_{\mathrm{intr}})=\{0\}.
\]
The zero intrinsic eigenvalue is tangent to this perimeter preserving rhombus
family. Thus the square is flat degenerate.
\begin{remark}[Flatness of contact hull cycles]
\label{rem:flat-contact-cycles}
The flatness in the square example above comes from the fact that every edge
of the fixed hull cycle is an active contact edge. Thus every hull edge length
remains equal to \(2\) along paths in the same fixed hull and contact class.
\end{remark}

\subsection{Rhombus with diagonal contact rigidity modulo rigid motions}
\label{ex:rhombus-rigid}
Consider
\[
{\bf c}_1=(0,0),\quad
{\bf c}_2=(2,0),\quad
{\bf c}_3=(3,\sqrt{3}),\quad
{\bf c}_4=(1,\sqrt{3}),
\]
with realised contact graph
\[
E(G)=\bigl\{\{1,2\},\{2,3\},\{3,4\},\{4,1\},\{2,4\}\bigr\}
\]
and hull cycle
\[
\mathcal B=\{(1,2),(2,3),(3,4),(4,1)\}.
\]
This example has the same hull cycle and the same perimeter as
Section~\ref{ex:square-degenerate}, but its reduced admissible space is
trivial. A direct computation gives
\[
A({\bf c})^\top\lambda=\nabla\widetilde\Per({\bf c}),
\qquad
\lambda=(1,1,1,1,0),
\]
in the displayed edge order, so \({\bf c}\) is class critical.
A direct computation also gives
\[
\operatorname{rank} A({\bf c})=5,
\qquad
\dim\Roll({\bf c})=8-5=3.
\]
Since \(R({\bf c})\subset\Roll({\bf c})\) and \(\dim R({\bf c})=3\),
\[
R({\bf c})=\Roll({\bf c}),
\qquad
\Roll({\bf c})\cap R({\bf c})^\perp=\{0\}.
\]
Thus there are no nontrivial reduced admissible directions. Unlike
Section~\ref{ex:square-degenerate}, no intrinsic Hessian has to be
diagonalised, and
\[
\operatorname{spec}(\mathcal H_{\mathrm{intr}})=\varnothing.
\]
Thus this rhombus is rigid modulo rigid motions.

\begin{table}[ht]
\centering
\caption{Variational data for the class critical four disk examples.}
\label{tab:four-disk-spectrum}
\renewcommand{\arraystretch}{1.18}
\begin{tabular}{p{0.34\textwidth}ccc}
\toprule
Data & Square & Rhombus with diagonal & Central star \\
\midrule
Class critical status
& yes
& yes
& yes \\
Class critical value of \(\Per\)
& \(8+2\pi\)
& \(8+2\pi\)
& \(6\sqrt{3}+2\pi\) \\
\(\operatorname{rank} A({\bf c})\)
& \(4\)
& \(5\)
& \(3\) \\
\(\dim\Roll({\bf c})=8-\operatorname{rank} A({\bf c})\)
& \(4\)
& \(3\)
& \(5\) \\
\(\dim\mathcal R_{\bf c}\)
& \(1\)
& \(0\)
& \(2\) \\
Intrinsic spectrum
& \(\{0\}\)
& \(\varnothing\)
& \(\left\{-\frac{3\sqrt{3}}{5},-\frac{3\sqrt{3}}{5}\right\}\) \\
Variational outcome
& flat
& rigid
& second order unstable \\
\bottomrule
\end{tabular}
\end{table}

\subsection{A central disk with three peripheral contacts}
\label{ex:central-three-peripheral}
Consider
\[
{\bf c}_1=(0,0),\quad
{\bf c}_2=(2,0),\quad
{\bf c}_3=(-1,\sqrt{3}),\quad
{\bf c}_4=(-1,-\sqrt{3}),
\]
with realised contact graph
\[
E(G)=\bigl\{\{1,2\},\{1,3\},\{1,4\}\bigr\}
\]
and hull cycle
\[
\mathcal B=\{(2,3),(3,4),(4,2)\}.
\]
This example gives a class critical configuration with genuine second order
descent directions. The three peripheral disks form an equilateral hull, while
the central disk is not a hull vertex. The noncontact distances between
peripheral centres are all \(2\sqrt{3}>2\), so the displayed contact graph is
exactly realised.
A direct computation gives
\[
A({\bf c})^\top\lambda=\nabla\widetilde\Per({\bf c}),
\qquad
\lambda=(\sqrt{3},\sqrt{3},\sqrt{3}),
\]
in the displayed edge order, so \({\bf c}\) is class critical.
A direct computation also gives
\[
\operatorname{rank} A({\bf c})=3,
\qquad
\dim\Roll({\bf c})=8-3=5.
\]
Since \(\dim R({\bf c})=3\) and \(R({\bf c})\subset\Roll({\bf c})\),
\[
\dim\bigl(\Roll({\bf c})\cap R({\bf c})^\perp\bigr)=2.
\]
After removing rigid motions, the remaining directions correspond to changing
the angular gaps between the three peripheral disks while preserving their
contacts with the central disk. With respect to an orthonormal basis of this two dimensional reduced admissible
space, a direct computation gives
\[
\operatorname{spec}(\mathcal H_{\mathrm{intr}})
=
\left\{-\frac{3\sqrt{3}}{5},-\frac{3\sqrt{3}}{5}\right\}.
\]
Thus the central configuration is class critical but second order unstable,
the intrinsic Hessian is negative definite on the two dimensional reduced
admissible space, so every nonzero reduced admissible direction is a second
order descent direction for the perimeter.

\subsection{The collinear path degeneration}
\label{ex:path-four}
Let \(X=\mathcal F\cap\mathcal C(G)\) be a fixed hull and contact class whose
contact graph is
\[
G=P_4,
\qquad
E(G)=\bigl\{\{1,2\},\{2,3\},\{3,4\}\bigr\},
\]
and whose nondegenerate hull cycle is
\[
\mathcal B=\{(1,2),(2,3),(3,4),(4,1)\}.
\]
This is the noncollinear fixed class with the displayed hull cycle. Collinear
realisations are boundary degenerations and are not interior regular points of
this fixed hull and contact class. The point of this example is to use the hull
leaf obstruction. The geometric contradiction is immediate once the relevant
collinearities are known, but those collinearities are forced by class
criticality through Theorem~\ref{thm:leaf-interior-descent}.
The vertices \(1\) and \(4\) are hull leaves. Suppose, for a contradiction,
that \({\bf c}\in X\) is relative interior regular and class critical, in the sense of
Proposition~\ref{prop:first order-criticality}. At the hull leaf \(1\), the
nodal perimeter gradient is
\[
{\bf g}_1={\bf t}_{41}-{\bf t}_{12}.
\]
By Theorem~\ref{thm:leaf-interior-descent}, this vector must be parallel to
\({\bf u}_{12}={\bf t}_{12}\). Hence \({\bf t}_{41}\) and
\({\bf t}_{12}\) are collinear. Similarly, at the hull leaf \(4\), the nodal
perimeter gradient is
\[
{\bf g}_4={\bf t}_{34}-{\bf t}_{41}.
\]
The same theorem forces this vector to be parallel to
\({\bf u}_{34}={\bf t}_{34}\). Hence \({\bf t}_{41}\) and
\({\bf t}_{34}\) are collinear. Therefore
\({\bf t}_{12},{\bf t}_{41},{\bf t}_{34}\) are all collinear, impossible for a
nondegenerate convex quadrilateral with cyclic hull edges
\[
(1,2),(2,3),(3,4),(4,1).
\]
Thus \(X\) contains no interior regular class critical configuration. This
class is excluded at the first order stage, so no class critical Hessian and no
intrinsic spectrum are assigned to this example.

\section{Computational verification of five disks }
\label{app:computational-verification}
This appendix records the computational certificate used in the finite verification
of Section~\ref{sec:n5}. The complete implementation is the Python script \texttt{supplementary\_verification.py} \\ ~\cite{repo}, which takes the dataset
\texttt{5disks.json} and generates all certificate data using interval arithmetic.

\subsection{Setup} For each class $X_i$, the input is the exact symbolic centre coordinate vector ${\bf c}$, the contact edge set $E(G)$, and the cyclic hull order $\mathcal{B}$. The analytical quantities perimeter gradient ${\bf g}$, contact operator $A({\bf c})$, and rigid motion space $R({\bf c})$ are computed using the definitions of Section~\ref{sec:fixed-class}. Integer-pivoted Gaussian elimination yields a basis $Z$ for $\Roll({\bf c})\cap R({\bf c})^\perp$.

\subsection{Classification} Each class is assigned one of four outcomes in order of priority: exclusion by hull leaf non parallelism (Theorem~\ref{thm:leaf-interior-descent}), exclusion by first order projected gradient descent (Proposition~\ref{prop:direct-first order-exclusion}), exclusion by admissible perturbation along $Z_{\cdot,1}$, certified by a negative pivot in the LDL$^\top$ decomposition of the intrinsic Hessian (Proposition~\ref{prop:second order-exclusion}), or survival. Class criticality for survivors is certified by bounding the first order multiplier residual below $10^{-50}$.

For a class with hull leaves, the verification does not use the displayed representative alone. It first reduces to the leaf parallel locus imposed by Theorem~\ref{thm:leaf-interior-descent}; the remaining candidates are then tested by admissible descent in the rolling space. Thus an exclusion
certificate for a leaf class means that both the nonparallel leaf locus and the residual leaf parallel locus have been excluded.



\end{document}